\documentclass{article}

\usepackage[utf8]{inputenc}

\usepackage{amssymb,amsmath,amscd,dsfont}
\usepackage[all,cmtip,line]{xy}
\newcommand{\op}[1]{\ensuremath{\operatorname{#1}}}
\newcommand{\wt}[1]{\ensuremath{\widetilde{#1}}}
\newcommand{\wh}[1]{\ensuremath{\widehat{#1}}}
\newcommand{\ol}[1]{\ensuremath{\overline{#1}}}
\newcommand{\ul}[1]{\ensuremath{\underline{#1}}}
\newcommand{\mc}[1]{\ensuremath{\mathcal{#1}}}
\newcommand{\fg}{\ensuremath{\mathfrak{g}}}
\newcommand{\mf}[1]{\ensuremath{\mathfrak{#1}}}
\newcommand{\B}{\ensuremath{\mathbb{B}}}
\newcommand{\R}{\ensuremath{\mathbb{R}}}
\newcommand{\N}{\ensuremath{\mathbb{N}}}
\newcommand{\Z}{\ensuremath{\mathbb{Z}}}
\newcommand{\bS}{\ensuremath{\mathbb{S}}}
\newcommand{\id}{\ensuremath{\operatorname{id}}}
\newcommand{\pr}{\ensuremath{\operatorname{pr}}}
\newcommand{\Map}{\ensuremath{\operatorname{Map}}}
\newcommand{\Hom}{\ensuremath{\operatorname{Hom}}}
\newcommand{\im}{\ensuremath{\operatorname{im}}}
\newcommand{\tx}[1]{\ensuremath{\text{#1}}}
\newcommand{\per}{\ensuremath{\operatorname{per}}}
\newcommand{\se}{\ensuremath{\nobreak\subseteq\nobreak}}
\newcommand{\from}{\ensuremath{\nobreak\colon\nobreak}}
\renewcommand{\to}{\ensuremath{\nobreak\rightarrow\nobreak}}
\newenvironment{tabsection}{}{}
\usepackage[a4paper,hscale=0.8,vscale=0.82,hmarginratio=1:1,includehead,headheight=14pt]{geometry}
\usepackage{fancyhdr}
\fancypagestyle{plain}{\fancyhead{}\fancyfoot[C]{\thepage}}

\usepackage{amssymb,amsmath,amscd}

\usepackage[amsmath,thmmarks]{ntheorem}
\theoremseparator{.}
\theoremsymbol{\rule{1ex}{1ex}}
\theorembodyfont{\upshape}
\newtheorem{definition}{Definition}[section]
\newtheorem{remark}[definition]{Remark}

\newtheorem{example}[definition]{Example}

\newtheorem*{proof}{Proof}

\theorembodyfont{\itshape}

\newtheorem*{nntheorem}{Theorem}
\theoremsymbol{}
\newtheorem{lemma}[definition]{Lemma}
\newtheorem{proposition}[definition]{Proposition}
\newtheorem{theorem}[definition]{Theorem}
\newtheorem{corollary}[definition]{Corollary}

\usepackage{graphicx}

\usepackage[all,cmtip,line]{xy}
\CompileMatrices

\usepackage[unicode]{hyperref}
\hypersetup{
    urlcolor = red,
	colorlinks = true,
	linkcolor = blue,
	citecolor = blue,
    linktocpage = true,
	pdftitle = {A Cocycle Model for Topological and Lie Group Cohomology},
    pdfauthor = {Friedrich Wagemann and Christoph Wockel},
    bookmarksopen = true,
    bookmarksopenlevel = 1,
	unicode = true,
    hypertexnames = false
  }

\newcommand{\dgp}{\ensuremath{\op{\mathtt{d}_{\mathrm{gp}}}}}
\newcommand{\kC}{\ensuremath{C_{\boldsymbol{k}}}}
\newcommand{\ktimes}[1][]{\ensuremath{\mathop{\times_{\boldsymbol{k}}^{#1}}}}
\newcommand{\ptimes}[1][]{\ensuremath{\mathop{\times_{\boldsymbol{p}}^{#1}}}}
\newcommand{\mtimes}[1][]{\ensuremath{\mathop{\times_{\boldsymbol{m}}^{#1}}}}
\newcommand{\superscript}[1]{\ensuremath{^{\textrm{#1}}}}
\newcommand{\Sh}{\ensuremath{\op{Sh}}}
\newcommand{\cont}{\ensuremath{c}}
\newcommand{\sm}{\ensuremath{s}}
\newcommand{\loc}{\ensuremath{  \op{loc}}}
\newcommand{\bc}{\ensuremath{ bc}}
\newcommand{\locc}{\ensuremath{ \op{loc},\cont}}
\newcommand{\loctop}{\ensuremath{ \op{loc},\op{top}}}
\newcommand{\globtop}{\ensuremath{ \op{glob},\op{top}}}
\newcommand{\locs}{\ensuremath{ \op{loc},\sm}}
\newcommand{\globc}{\ensuremath{\op{glob},\cont}}
\newcommand{\globs}{\ensuremath{\op{glob},\sm}}
\newcommand{\simpc}{\ensuremath{\op{simp},\cont}}
\newcommand{\simps}{\ensuremath{\op{simp},\sm}}
\newcommand{\meas}{\ensuremath{\op{\mu}}}
\newcommand{\SM}{\ensuremath{\op{SM}}}

\newcommand{\Moore}{\ensuremath{\op{Moore}}}

\newcommand{\cat}[1]{\ensuremath{\boldsymbol{\op{#1}}}}

\newcommand{\cghaus}{\cat{kTop}}

\usepackage[normalem]{ulem}
\def\dashuline{\bgroup 
  \ifdim\ULdepth=\maxdimen  %
   \settodepth\ULdepth{(j}\advance\ULdepth.4pt\fi
  \markoverwith{\kern.1em
  \vtop{\kern\ULdepth \hrule width .2em}%
  \kern.1em}\ULon}

\begin{document}
\title{A Cocycle Model for Topological and Lie Group Cohomology}

\author{Friedrich Wagemann \\
     wagemann@math.univ-nantes.fr\\
\and
     Christoph Wockel\\
     christoph@wockel.eu}

\maketitle

\begin{abstract}
 We propose a unified framework in which the different constructions of
 cohomology groups for topological and Lie groups can all be
 treated on equal footings. In particular, we show that the cohomology
 of ``locally continuous'' cochains (respectively ``locally smooth'' in
 the case of Lie groups) fits into this framework, which provides an
 easily accessible cocycle model for topological and Lie group cohomology.
 We illustrate the use of this unified framework and the relation
 between the different models in various applications. This includes the
 construction of cohomology classes characterizing the string group and
 a direct connection to Lie algebra cohomology.
\end{abstract}

\begin{tabular}{rp{35em}}
\textbf{MSC:} & 22E41 (primary), 57T10, 20J06 (secondary)\\
\textbf{Keywords:} & cohomology for topological groups, cohomology for
Lie groups, abelian extension, crossed module, Lie algebra cohomology,
string group
\end{tabular}

\section*{Introduction}

\begin{tabsection}
 It is a common pattern in mathematics that things that are easy to define are
 hard to compute and things that are hard to define come with lots of machinery
 to compute them\footnote{Quote taken from a lecture by Janko Latschev.}. On
 the other hand, mathematics can be very enjoyable if these different
 definitions can be shown to yield isomorphic objects. In the present article
 we want to promote such a perspective towards topological group cohomology,
 along with its specialization to Lie group cohomology.
 
 It has become clear in the last decade that concretely accessible cocycle
 models for cohomology theories (understood in a broader sense) are as
 important as abstract constructions. Examples for this are differential
 cohomology theories (cocycle models come for instance from (bundle) gerbes, an
 important concept in topological and conformal field theory), elliptic
 cohomology (where cocycle models are yet conjectural but have nevertheless
 already been quite influential) and Chas-Sullivan's string topology operations
 (which are subject to certain well behaved representing cocycles). This
 article describes an easily accessible cocycle model for the more complicated
 to define cohomology theories of topological and Lie groups
 \cite{Segal70Cohomology-of-topological-groups,Wigner73Algebraic-cohomology-of-topological-groups,Deligne74Theorie-de-Hodge.-III,Brylinski00Differentiable-Cohomology-of-Gauge-Groups}.
 The cocycle model is a seemingly obscure mixture of (abstract) group
 cohomology, added in a continuity condition only around the identity. Its
 smooth analogue has been used in the context of Lie group cohomology and its
 relation to Lie algebra cohomology
 \cite{TuynmanWiegerinck87Central-extensions-and-physics,WeinsteinXu91Extensions-of-symplectic-groupoids-and-quantization,Neeb02Central-extensions-of-infinite-dimensional-Lie-groups,Neeb04Abelian-extensions-of-infinite-dimensional-Lie-groups,Neeb06Towards-a-Lie-theory-of-locally-convex-groups,Neeb07Non-abelian-extensions-of-infinite-dimensional-Lie-groups},
 which is where our original motivation stems from. The basic message will be
 that all the above concepts of topological and Lie group cohomology coincide
 for finite-dimensional Lie groups and coefficients modeled on quasi-complete
 locally convex spaces. Beyond finite-dimensional Lie groups still all
 continuous concepts agree.

 There is a simple notion of topological group cohomology for a topological
 group $G$ and a continuous $G$-module $A$. It is the cohomology of the complex
 of continuous cochains with respect to the usual group differential. This is
 what we call ``globally continuous'' group cohomology and denote it by
 $H^{n}_{\globc}(G,A)$. It cannot encode the topology of $G$ appropriately, for
 instance $H^{2}_{\globc}(G,A)$ only describes abelian extensions which are
 topologically trivial bundles. However, in case $G$ is contractible it will
 turn out that the more elaborate cohomology groups from above coincide with
 $H^{n}_{\globc}(G,A)$. In this sense, the deviation from the above cohomology
 groups from being the globally continuous ones measures the non-triviality of
 the topology of $G$. On the other hand, the comparison between
 $H^{n}_{\globc}(G,A)$ and the other cohomology theories for topologically
 trivial \emph{coefficients} $A$ will lead to a comparison theorem between
 the other cohomology theories. It is this circle of ideas that the present
 article is about.\\
 
 The paper is organized as follows. In the first section we review the
 construction and provide the basic facts of what we call locally continuous
 group cohomology $H^{n}_{\locc}(G,A)$ (respectively locally smooth cohomology
 $H^{n}_{\locs}(G,A)$ for $G$ a Lie group and $A$ a smooth $G$-module). Since
 it will become important in the sequel we highlight in particular that for
 loop contractible\footnote{$A$ is called \emph{loop-contractible} if there
 exists a contracting homotopy $\rho\from [0,1]\times A\to A$ such that
 $\rho_{t}\from A\to A$ is a group homomorphism for each $t\in [0,1]$.}
 coefficients these cohomology groups coincide with the globally continuous
 (respectively smooth) cohomology groups $H^{n}_{\globc}(G,A)$ (respectively
 $H^{n}_{\globs}(G,A)$). In the second section we then introduce what we call
 simplicial continuous cohomology $H^{n}_{\simpc}(G,A)$ and construct a
 comparison morphism $H^{n}_{\simpc}(G,A)\to H^{n}_{\locc}(G,A)$. The third
 section explains how simplicial cohomology may be computed in a way similar to
 computing sheaf cohomology via \v{C}ech cohomology (the fact that this gives
 indeed $H^{n}_{\simpc}(G,A)$ will have to wait until the next section).
 
 The first main point of this paper comes in Section
 \ref{sect:Comparison_Theorem}, where we give the following axiomatic
 characterization of what we call a cohomology theory for topological groups.
 \begin{nntheorem}[Comparison
  Theorem] Let $G$ be a compactly generated topological group and let
  $\cat{G-Mod}$ be the category of locally contractible $G$-modules. Then there
  exists, up to isomorphism, exactly one sequence of functors
  $(H^{n}\from\cat{G-Mod}\to\cat{Ab})_{n\in\N_{0}}$ admitting natural long
  exact sequences for short exact sequences in $\cat{G-Mod}$ such that
  \begin{enumerate}
   \item $H^{0}(A)=A^{G}$ is the invariants functor
   \item $H^{n}(A)=H^{n}_{\globc}(G,A)$ for contractible $A$.
  \end{enumerate}
 \end{nntheorem}
 There is one other way of defining cohomology groups $H^{n}_{\SM}(G,A)$ which
 is due to Segal and Mitchison \cite{Segal70Cohomology-of-topological-groups}.
 This construction will turn out to be the one which is best suited for
 establishing the Comparison Theorem. However, we then show that under some
 mild assumptions (guaranteed for instance by the metrizability of $G$) all
 cohomology theories that we had so far (except the globally continuous) obey
 these axioms. The rest of the section in then devoted to showing that
 almost all other concepts of cohomology theories for topological groups also
 fit into this scheme. This includes the ones considered by Flach in
 \cite{Flach08Cohomology-of-topological-groups-with-applications-to-the-Weil-group},
 the measurable cohomology of Moore from
 \cite{Moore76Group-extensions-and-cohomology-for-locally-compact-groups.-III}
 and the mixture of measurable and locally continuous cohomology of Khedekar and
 Rajan from \cite{KhedekarRajan12On-cohomology-theory-for-topological-groups}.
 The only exception that we know not fitting into this scheme is the continuous
 bounded cohomology (see
 \cite{Monod01Continuous-bounded-cohomology-of-locally-compact-groups,Monod06An-invitation-to-bounded-cohomology}),
 which differs from the above concepts by design.
 
 The second main point comes with Section \ref{sect:examples}, where we exploit
 the interplay between the different constructions. For instance, we construct
 a cohomology class that deserves to be named string class, and we construct
 topological crossed modules associated to third cohomology classes. Moreover,
 we show how to extract the purely topological information contained in an
 element in $H^{n}_{\locc}(G,A)$ by relating an explicit formula for this with
 a structure map for the spectral sequence associated to $H^{n}_{\simpc}(G,A)$.
 Furthermore, $H^{n}_{\locs}(G,A)$ maps naturally to Lie algebra cohomology and
 we use the previous result to identify situations where this map becomes an
 isomorphism. Almost none of the consequences mentioned here could be drawn
 from one model on its own, so this demonstrates the strength of the unified
 framework.
 
 In the last two sections, which are independent from the rest of the paper, we
 provide some details on the constructions that we use.
\end{tabsection}

\section*{Acknowledgements}

\begin{tabsection}
 Major parts of the research on this article were done during research
 visits of CW in Nantes (supported by the program MatPyL ``Math\'ematiques en
 Pays de la Loire'') and of FW in Hamburg (supported by the University of
 Hamburg). Both authors want to thank Chris Schommer-Pries for various
 enlightening discussions and conversations, which influenced the paper
 manifestly (see \cite{Pries09Smooth-group-cohomology} for a bunch of
 ideas pointing towards the correct direction). Thanks go also to Rolf Farnsteiner for pointing
 out the usefulness of the language of $\delta$-functors.
\end{tabsection}

\section*{Conventions}

\begin{tabsection}
 Since we will be working in the two different regimes of compactly generated
 Hausdorff spaces and infinite-dimensional Lie groups we have to choose the
 setting with some care.
 
 Unless further specified, $G$ will throughout be a group in the category
 $\cghaus$ of $k$-spaces (compactly generated Hausdorff\footnote{More
 generally, our results remain valid if one only considers weak Hausdorff
 spaces.} spaces i.e., a subset is closed if and only if its intersection with
 each compact set is closed, cf.\
 \cite{Whitehead78Elements-of-homotopy-theory,Mac-Lane98Categories-for-the-working-mathematician}
 or \cite{Hovey99Model-categories}) and $A$ will be a (locally contractible)
 $G$-module in this category\footnote{From the beginning of Section
 \ref{sect:Comparison_Theorem} we will need that $A$ is locally
 contractible.}. This means that the multiplication (respectively action) map
 is continuous with respect to the \emph{compactly generated topology} on the
 product. Note that the topology on the product may be finer than the product
 topology, so this may not be a topological group (respectively module) as
 defined below. To avoid confusion, we denote the compactly generated product
 by $X\ktimes Y$ (and $X^{\ktimes[n]}$ for the $n$-fold product) and the
 compactly generated topology on $C(X,Y)$ by $\kC(X,Y)$ for $X,Y$ in $\cghaus$.
 
 If $X$ and $Y$ are arbitrary topological spaces, then we refer to the product
 topology by $X\ptimes Y$ (and $X^{\ptimes[n]}$). By topological group
 (respectively topological module) we shall mean a group (respectively module)
 in this category, i.e., the multiplication (respectively action) is continuous
 for the product topology.
 
 Frequently we will assume, in addition, that $G$ is a (possibly infinite
 dimensional) Lie group and that $A$ is a smooth $G$-module\footnote{This
 assumption seems to be quite restrictive for either side, but it is the
 natural playground on which homotopy theory and (infinite-dimensional) Lie
 theory interacts.}. By this we mean that $G$ is a group in the category
 $\cat{Man}$ of manifolds, modeled on locally convex vector spaces (see
 \cite{Hamilton82The-inverse-function-theorem-of-Nash-and-Moser,Milnor84Remarks-on-infinite-dimensional-Lie-groups,Neeb06Towards-a-Lie-theory-of-locally-convex-groups}
 or \cite{GlocknerNeeb13Infinite-dimensional-Lie-groups} for the precise
 setting) and $A$ is a $G$-module in this category. This means in particular
 that the multiplication (respectively action) map is smooth for the product
 smooth structure. To avoid confusion we refer to the product in $\cat{Man}$ by
 $X\mtimes Y$ (and $X^{\mtimes[n]}$).
 
 Note that we set things up in such a way that the smooth setting is a
 specialization of the topological one, which is in turn a specialization of
 the compactly generated one. This is true since smooth maps are in particular
 continuous and since the product topology is coarser than the compactly
 generated one. Note also that all topological properties on $G$ (except the
 existence of good covers) that we will assume are satisfied for metrizable $G$
 and all smoothness properties are satisfied for metrizable and smoothly
 paracompact $G$. The existence of good covers (as well as metrizability and
 smooth paracompactness) is in turn satisfied for large classes of
 infinite-dimensional Lie groups like mapping groups or diffeomorphism groups
 \cite{KrieglMichor97The-Convenient-Setting-of-Global-Analysis,SperaWurzbacher10Good-coverings-for-section-spaces-of-fibre-bundles}.
 
 We shall sometimes have to impose topological conditions on the topological
 spaces $|G|$ and $|A|$ underlying $G$ and $A$. We will do so by leisurely
 adding the corresponding adjective. For instance, a contractible $G$-module
 $A$ is a $G$-module such that $|A|$ is contractible.
\end{tabsection}

\section{Locally continuous and locally smooth cohomology}
\label{sect:locally_smooth_cohomology}

\begin{tabsection}
 One of our main objectives will be the relation of locally continuous
 and locally smooth cohomology for topological or Lie groups to other
 concepts of topological group cohomology. In this section, we recall
 the basic notions and properties of locally continuous and locally
 smooth cohomology. These concepts already appear in the work of
 Tuynman-Wiegerinck
 \cite{TuynmanWiegerinck87Central-extensions-and-physics}, of
 Weinstein-Xu
 \cite{WeinsteinXu91Extensions-of-symplectic-groupoids-and-quantization}
 and have been popularized recently by Neeb
 \cite{Neeb02Central-extensions-of-infinite-dimensional-Lie-groups,Neeb04Abelian-extensions-of-infinite-dimensional-Lie-groups,Neeb06Towards-a-Lie-theory-of-locally-convex-groups,Neeb07Non-abelian-extensions-of-infinite-dimensional-Lie-groups}.
 There has also appeared a slight variation of this by measurable locally smooth
 cohomology in
 \cite{KhedekarRajan12On-cohomology-theory-for-topological-groups}.
\end{tabsection}

\begin{definition}\label{def:locsm}
 For any pointed topological space $(X,x)$ 	and abelian topological
 group $A$ we set
 \begin{align*}
  C_{\loc}(X,A):=\{f\from X\to A\mid f\text{ is continuous on some neighborhood of }x\}.
 \end{align*}
 If, moreover, $X$ is a smooth manifold and $A$ a Lie group, then we set
 \begin{align*}
  C_{\loc}^{\infty}(X,A):=\{f\from X\to A\mid f\text{ is smooth on some neighborhood of }x\}.
 \end{align*}
 With this we set $C^{n}_{\locc}(G,A):=C_{\loc}(G^{\ktimes[n]},A)$,
 where we choose the identity in $G^{n}$ as base-point. We call these
 functions (by some abuse of language) \emph{locally continuous group
 cochains}. The ordinary group differential
 \begin{align}
  (\dgp f)(g_0,\ldots,g_n)&=g_0.f(g_1,\ldots,g_n)\,+\notag\\  
  &+\sum_{j=1}^n(-1)^jf(g_0,\ldots,g_{j-1}g_j,\ldots,g_n)\,+\,(-1)^{n+1}
  f(g_0,\ldots,g_{n-1})\label{eqn:group_differential}
 \end{align}
 turns $(C_{\locc}^{n}(G,A),\dgp)$ into a cochain complex. Its
 cohomology will be denoted by $H^{n}_{\locc}(G,A)$ and be called the
 \emph{locally continuous group cohomology}.

 If $G$ is a Lie group and $A$ a smooth $G$-module, then we also
 consider the sub complex
 $C^{n}_{\locs}(G,A):=C_{\loc}^{\infty}(G^{\mtimes[n]},A)$ and call its
 cohomology $H^{n}_{\locs}(G,A)$ the \emph{locally smooth group
 cohomology}.
\end{definition}

\begin{tabsection}
 These two concepts should not be confused with the continuous
 \emph{local cohomology} (respectively the smooth \emph{local
 cohomology}) of $G$, which is given by the complex of germs of
 continuous (respectively smooth) $A$-valued functions at the identity
 (which is isomorphic to the Lie algebra cohomology for a
 finite-dimensional Lie group $G$, see Remark
 \ref{rem:connection_to_Lie_algebra_cohomology}). It is crucial that the
 cocycles in the locally continuous cohomology actually are extensions
 of locally defined cocycles and this extension is extra information
 they come along with. Note for instance, that not all locally defined
 homomorphisms of a topological groups extend to global homomorphisms
 and that not all locally defined 2-cocycles extend to globally defined
 cocycles
 \cite{Smith51The-complex-of-a-group-relative-to-a-set-of-generators.-I,Smith51The-complex-of-a-group-relative-to-a-set-of-generators.-II,Est62Local-and-global-groups.-I,Est62Local-and-global-groups.-II}.
\end{tabsection}

\begin{remark}\label{rem:long_exact_coefficient_sequence}
 (cf.\ \cite[App.\
 E]{Neeb04Abelian-extensions-of-infinite-dimensional-Lie-groups}) Let
 \begin{equation}
  A\xrightarrow{\alpha}B\xrightarrow{\beta}C \label{eqn:shot_exact_coefficient_sequence}
 \end{equation}
 be a \emph{short exact sequence of $G$-modules in $\cghaus$},
 i.e., the underlying sequence of abstract abelian groups is exact and
 $\beta$ (or equivalently $\alpha$) has a continuous local section. The
 latter is equivalent to demanding that
 \eqref{eqn:shot_exact_coefficient_sequence} is a locally trivial
 principal $A$-bundle. Then composition with $\alpha$ and $\beta$
 induces a sequence
 \begin{equation}\label{eqn:short_exact_sequence_of_complexes}
  C^{n}_{\locc}(G,A)\xrightarrow{\alpha_{*}} C^{n}_{\locc}(G,B)\xrightarrow{\beta _{*}}
  C^{n}_{\locc}(G,C),
 \end{equation}
 which we claim to be a short exact sequence of chain complexes.
 Injectivity of $\alpha_{*}$ and $\im(\alpha_{*})\se\ker(\beta_{*})$ is
 clear. Since a local trivialization of the bundle induces a continuous
 left inverse to $\alpha$ on some neighborhood of $\ker(\beta)$, we also
 have $\ker(\beta_{*})\se \im(\alpha_{*})$. To see that $\beta_{*}$ is
 surjective, we choose a local continuous section $\sigma\from U\to B$
 which we extend to a global (but not necessarily continuous) section
 $\sigma\from C\to B$. Thus if $f\in C^{n}_{\locc}(G,C)$, then
 $  \sigma \circ f \in C^{n}_{\locc}(G,B)$ with
 $\beta_{*}(\sigma \circ f)=\beta \op{\circ} \sigma \op{\circ} f=f$ and
 $\beta_{*}$ is surjective.
 Since \eqref{eqn:short_exact_sequence_of_complexes} is exact, it induces
 a long exact sequence
 \begin{equation}\label{eqn:long_exact_coefficient_sequence}
  \cdots\to H^{n-1}_{\locc}(G,C) \to H^{n}_{\locc}(G,A)\to H^{n}_{\locc}(G,B) \to H^{n}_{\locc}(G,C) \to H^{n+1}_{\locc}(G,A) \to\cdots
 \end{equation}
 in the locally continuous cohomology.

 If, in addition, $G$ is a Lie group and
 \eqref{eqn:shot_exact_coefficient_sequence} is a \emph{short exact
 sequence of smooth $G$-modules}, i.e., a smooth locally trivial
 principal $A$-bundle, then the same argument shows that $\alpha_{*}$
 and $\beta_{*}$ induce a long exact sequence
 \begin{equation*}
  \cdots\to H^{n-1}_{\locs}(G,C) \to H^{n}_{\locs}(G,A)\to H^{n}_{\locs}(G,B) \to H^{n}_{\locs}(G,C) \to H^{n+1}_{\locs}(G,A) \to\cdots
 \end{equation*}
 in the locally smooth cohomology.
\end{remark}

\begin{remark}
 The low-dimensional cohomology groups $H^{0}_{\locc}(G,A)$,
 $H^{1}_{\locc}(G,A)$ and $H^{2}_{\locc}(G,A)$ have the usual interpretations.
 $H^{0}_{\locc}(G,A)=A^{G}$ are the $G$-invariants of $A$, $H^1_{\locc}(G,A)$
 (respectively $H^{1}_{\locs}(G,A)$) is the group of equivalence classes of
 continuous (respectively smooth) crossed homomorphisms modulo principal
 crossed homomorphisms. If $G$ is connected\footnote{The requirement on $G$
 being connected is a posteriori redundant, since the isomorphism also follows
 from the comparison result in Section \ref{sect:Comparison_Theorem} and
 \cite[\S 4]{Segal70Cohomology-of-topological-groups}. However, the argument
 given in \cite[Sect.\
 2]{Neeb04Abelian-extensions-of-infinite-dimensional-Lie-groups} requires
 connectedness. It would be interesting to have an argument similar to the one
 from \cite[Sect.\
 2]{Neeb04Abelian-extensions-of-infinite-dimensional-Lie-groups} (i.e., using
 only locally continuous group cocycles) also in the non-connected case (see
 also the concept of a \emph{strongly smooth outer action} in \cite[Sect.\
 1.2]{Neeb07Non-abelian-extensions-of-infinite-dimensional-Lie-groups}).}, then
 $H^{2}_{\locc}(G,A)$ (respectively $H^{2}_{\locs}(G,A)$) is isomorphic to the
 group of equivalence classes of abelian extensions
 \begin{equation}\label{eqn:locally_trivial_abelian_extension}
  A\to \wh{G}\to G
 \end{equation}
 which are continuous (respectively smooth) locally trivial principal
 $A$-bundles over $G$ \cite[Sect.\
 2]{Neeb04Abelian-extensions-of-infinite-dimensional-Lie-groups}.
\end{remark}

\begin{remark}\label{rem:globally_continuous_cohomology}
 The cohomology groups $H^{n}_{\locc}(G,A)$ and $H^{n}_{\locs}(G,A)$ are
 variations of the \emph{globally continuous} cohomology
 groups $H^{n}_{\globc}(G,A)$ and \emph{globally smooth} cohomology
 groups $H^{n}_{\globs}(G,A)$, which are the cohomology groups of the chain
 complexes
 \begin{equation*}
  C^{n}_{\globc}(G,A):=C(G^{\ktimes[n]},A)
  \quad\text{ and }\quad
  C^{n}_{\globs}(G,A):=C^{\infty}(G^{\mtimes[n]},A),
 \end{equation*}
 endowed with the differential \eqref{eqn:group_differential}. We
 obviously have
 \begin{equation*}
  H^{0}_{\locc}(G,A)=H^{0}_{\globc}(G,A)\quad\text{ and }\quad
  H^{0}_{\locs}(G,A)=H^{0}_{\globs}(G,A).
 \end{equation*}
 Since crossed homomorphisms are continuous (respectively smooth) if and
 only if they are so on some identity neighborhood (see for example
 \cite[Lemma III.1]
 {Neeb04Abelian-extensions-of-infinite-dimensional-Lie-groups}), we also have
 \begin{equation*}
  H^{1}_{\locc}(G,A)=H^{1}_{\globc}(G,A)\quad\text{ and }\quad
  H^{1}_{\locs}(G,A)=H^{1}_{\globs}(G,A).
 \end{equation*}
 Moreover, the argument from Remark
 \ref{rem:long_exact_coefficient_sequence} also shows that we have a
 long exact sequence
 \begin{equation*}
  \cdots\to H^{n-1}_{\globc}(G,C) \to H^{n}_{\globc}(G,A)\to H^{n}_{\globc}(G,B) \to H^{n}_{\globc}(G,C) \to H^{n+1}_{\globc}(G,A) \to\cdots
 \end{equation*}
 if the exact sequence $A\xrightarrow{\alpha} B\xrightarrow{\beta} C$
 has a global continuous section (and respectively for the globally
 smooth cohomology if $A\xrightarrow{\alpha} B\xrightarrow{\beta} C$ has
 a global smooth section).

 Now assume that $A$ is contractible (respectively smoothly
 contractible) and that $G$ is connected and paracompact (respectively smoothly
 paracompact). In this case, the bundle
 \eqref{eqn:locally_trivial_abelian_extension} has a global continuous
 (respectively smooth) section and thus the extension
 \eqref{eqn:locally_trivial_abelian_extension} has a representative in
 $H^{2}_{\globc}(G,A)$ (respectively $H^{2}_{\globs}(G,A)$), cf.\
 \cite[Prop.\
 6.2]{Neeb04Abelian-extensions-of-infinite-dimensional-Lie-groups}.
 Moreover, the argument in \cite[Prop.\
 6.2]{Neeb04Abelian-extensions-of-infinite-dimensional-Lie-groups} also
 shows that two extensions of the form
 \eqref{eqn:locally_trivial_abelian_extension} are in this case
 equivalent if and only if the representing globally continuous
 (respectively smooth) cocycles differ by a globally continuous
 (respectively smooth) coboundary, and thus the canonical homomorphisms
 \begin{equation*}
  H^{2}_{\globc}(G,A)\to H^{2}_{\locc}(G,A) \quad\text{ and }\quad
  H^{2}_{\globs}(G,A)\to H^{2}_{\locs}(G,A)
 \end{equation*}
 are isomorphisms in this case.
\end{remark}

\begin{tabsection}
 It will be crucial in the following that the latter observation
 also holds for a large
 class of contractible coefficients in arbitrary dimension (and in the
 topological case also for not necessarily paracompact $G$). For this,
 recall that $A$ is called \emph{loop-contractible} if there exists a
 contracting homotopy $\rho\from [0,1]\times A\to A$ such that
 $\rho_{t}\from A\to A$ is a group homomorphism for each $t\in [0,1]$.
 If $A$ is a Lie group, then it is called \emph{smoothly
 loop-contractible} if $\rho$ is, in addition, smooth. In particular,
 vector spaces are smoothly loop-contractible, but in the topological
 case there exist more elaborate and important examples (see Section
 \ref{sect:Comparison_Theorem}).
\end{tabsection}

\begin{proposition}\label{prop:loc=glob_for_contractible_coefficients}
 If $A$ is loop-contractible, and the product topology on all $G^{n}$ is
 compactly generated, then the inclusion
 $C^{n}_{\globc}(G,A)\hookrightarrow C^{n}_{\locc}(G,A)$ induces an
 isomorphism $H^{n}_{\globc}(G,A)\cong H^{n}_{\locc}(G,A)$.

 If $G$ is a Lie group such that all $G^{\mtimes[n]}$ are smoothly
 paracompact and $A$ is a smooth $G$-module which is smoothly
 loop-contractible, then
 $C^{n}_{\globs}(G,A)\hookrightarrow C^{n}_{\locs}(G,A)$ induces an
 isomorphism $H^{n}_{\globs}(G,A)\cong H^{n}_{\locs}(G,A)$.
\end{proposition}

\begin{proof}
 This is \cite[Prop.\ III.6, Prop.\
 IV.6]{FuchssteinerWockel11Topological-Group-Cohomology-with-Loop-Contractible-Coefficients}.
\end{proof}

In the
case of discrete $A$ we note that there is no difference
between the locally continuous and locally smooth cohomology groups.
This is immediate since continuous and smooth maps
into discrete spaces are both the same thing as constant maps on connected
components.

\begin{lemma}
 If $G$ is a Lie group and $A$ is a discrete $G$-module, then the
 inclusion $C_{\locs}^{n}(G,A)\hookrightarrow C_{\locc}^{n}(G,A)$
 induces an isomorphism in cohomology
 $H_{\locs}^{n}(G,A)\to H_{\locc}^{n}(G,A)$.
\end{lemma}

In the finite-dimensional case, we also note that there is no difference
between the locally continuous and locally smooth cohomology groups.

\begin{proposition}\label{prop:locc=locs_in_finite_dimensions}
 Let $G$ be a finite-dimensional Lie group, $\mathfrak{a}$ be a
 quasi-complete locally convex space\footnote{A locally convex space is
 said to be quasi-complete if each bounded Cauchy net converges.} on
 which $G$ acts smoothly, $\Gamma\se\mathfrak{a}$ be a discrete
 submodule and set $A=\mathfrak{a}/\Gamma$. Then the inclusion
 $C_{\locs}^{n}(G,A)\hookrightarrow C_{\locc}^{n}(G,A)$ induces an
 isomorphism $H_{\locs}^{n}(G,A)\cong H_{\locc}^{n}(G,A)$.
\end{proposition}

\begin{proof}
 (cf.\ \cite[Cor.\
 V.3]{FuchssteinerWockel11Topological-Group-Cohomology-with-Loop-Contractible-Coefficients})
 If $\Gamma=\{0\}$, then this is implied by Proposition
 \ref{prop:loc=glob_for_contractible_coefficients} and \cite[Thm.\
 5.1]{HochschildMostow62Cohomology-of-Lie-groups}. The general case then
 follows from the previous lemma, the short exact sequence for the
 coefficient sequence $\Gamma\to \mathfrak{a}\to A$ and the Five Lemma.
\end{proof}

\begin{remark}\label{rem:alternative-locally-continuous-cohomology}
 For a topological group $G$ and a topological $G$-module $A$ there also
 exists a variation of the locally continuous group cohomology, which
 are the cohomology groups of the cochain complex
 $(C_{\locc}(G^{\ptimes[n]},A),\dgp)$ (note the difference in the
 topology that we put on $G^{n}$). We denote this by
 $H^{n}_{\loctop}(G,A)$. The same argument as above yields long exact
 sequences from short exact sequences of topological $G$-modules that
 are locally trivial principal bundles. Moreover, they coincide with the
 corresponding globally continuous cohomology groups
 $H^{n}_{\globtop}(G,A)$ of $(C(G^{\ptimes[n]},A),\dgp)$ if $A$ is loop
 contractible \cite[Cor.\
 II.8]{FuchssteinerWockel11Topological-Group-Cohomology-with-Loop-Contractible-Coefficients}.
 We will use these cohomology groups very seldomly.
\end{remark}

\section{Simplicial group cohomology}
\label{sect:simplicial_group_cohomology}

\begin{tabsection}
 The cohomology groups that we introduce in this section date back to
 \cite[Sect.\ 3]{Wigner73Algebraic-cohomology-of-topological-groups} and
 have also been worked with for instance in
 \cite{Deligne74Theorie-de-Hodge.-III,Friedlander82Etale-homotopy-of-simplicial-schemes,Brylinski00Differentiable-Cohomology-of-Gauge-Groups,Conrad03Cohomological-Descent}.
 Since the simplicial cohomology groups are defined in terms of sheaves
 on simplicial space, we first recall some facts about it. The material
 is largely taken from
 \cite{Deligne74Theorie-de-Hodge.-III,Friedlander82Etale-homotopy-of-simplicial-schemes}
 and \cite{Conrad03Cohomological-Descent}.
\end{tabsection}

\begin{definition}
 Let $X_{\bullet}\from \Delta^{\op{op}}\to \cat{Top}$ be a simplicial
 space, i.e., a collection of topological spaces $(X_{k})_{k\in\N_{0}}$,
 together with continuous face maps
 $d_{k}^{i}\from X_{k}\to X_{k-1}$ for $i=0,\ldots,k$ and continuous
 degeneracy maps $s_{k}^{i}\from X_{k}\to X_{k+1}$ for $i=0,\ldots,k$
 satisfying the simplicial identities
 \begin{equation*}
  \begin{cases}
  d_{k}^{i}\op{\circ}d_{k+1}^{j}=d_{k}^{j-1}\op{\circ}d_{k+1}^{i}& \text{if }i<j\\
  d_{k+1}^{i}\op{\circ}s_{k}^{j}=s^{j-1}_{k-1}\op{\circ}d^{i}_{k} & \text{if }i<j\\
  d_{k+1}^{i}\op{\circ}s_{k}^{j}=  \id_{X_{k}} &\text{if }i=j\tx{ or }i=j+1\\
  d_{k+1}^{i}\op{\circ}s_{k}^{j}=  s^{{j}}_{k-1}\op{\circ}d^{i-1}_{k} & \text{if }i>j+1\\
  s^{i}_{k+1}\op{\circ}s^{j}_{k}=s^{j+1}_{k+1}\op{\circ}s^{i}_{k} & \text{if }i\leq j.
  \end{cases}
 \end{equation*}
 (cf.\ \cite{GoerssJardine99Simplicial-homotopy-theory}). Then a sheaf
 $E^{\bullet}$ on $X_{\bullet}$ consists of sheaves $E^{k}$ of abelian
 groups on each space $X_{k}$ and a collection of morphisms
 $D^{k}_{i}\from {d_{k}^{i}}^{*}E^{k-1}\to E^{k}$ (for $k\geq 1$) and
 $S^{k}_{i}\from {s_{k}^{i}}^{*}E^{k+1}\to E^{k}$, obeying the
 simplicial identities
 \begin{equation*}
  \left\{  \begin{array}{r@{\,}c@{\,}ll}
  D^{k}_{j}\op{\circ}{d_{k}^{j}}^{*}D^{k+1}_{i}&=&D^{k}_{i}\op{\circ}{d_{k}^{i}}^{*}D^{k+1}_{j-1} &\text{ if }i<j\\
  S^{k+1}_{j}\op{\circ}{s_{k+1}^{j}}^{*}D^{k}_{i}&=&D^{i}_{k+1}\op{\circ}{d_{k+1}^{i}}^{*}S^{k+2}_{j-1}  &\text{ if }i<j\\
  S^{k+1}_{j}\op{\circ}{s_{k+1}^{j}}^{*}D^{k}_{i}&=&  \id_{E^{k}} &\text{ if }i=j\tx{ or }i=j+1\\
  S^{k+1}_{j}\op{\circ}{s_{k+1}^{j}}^{*}D^{k}_{i}&=&D^{i-1}_{k+1}\op{\circ}{d_{k+1}^{i-1}}^{*}S^{{j}}_{k}  &\text{ if }i>j+1\\
  S^{k+1}_{j}\op{\circ}{s_{k}^{j}}^{*}S^{k}_{i}&=&S^{k+1}_{i}\op{\circ}{s_{k+1}^{i}}^{*}S^{k}_{j+1}  &\text{ if }i\leq j.
  \end{array}\right.
 \end{equation*}
 A morphism of sheaves $u\from E^{\bullet}\to F^{\bullet}$ consists of
 morphisms $u^{k}\from E^{k}\to F^{k}$ compatible with $D^{k}_{i}$ and
 $S^{k}_{i}$ (cf.\ \cite[5.1]{Deligne74Theorie-de-Hodge.-III}).
\end{definition}

\begin{tabsection}
 Note that $E^{\bullet}$ is \emph{not} what one usually would call a
 simplicial sheaf since the latter usually refers to a sheaf (on some
 arbitrary site) with values in simplicial sets or, equivalently, to a
 simplicial object in the category of sheaves (again, on some arbitrary
 site). However, one can interpret sheaves on $X_{\bullet}$ as sheaves
 on a certain site \cite[5.1.8]{Deligne74Theorie-de-Hodge.-III},
 \cite[Def.\ 6.1]{Conrad03Cohomological-Descent}.
\end{tabsection}

\begin{remark}
 Sheaves on $X_{\bullet}$ and their morphisms constitute a category
 $\Sh(X_{\bullet})$. Since morphisms in $\Sh(X_{\bullet})$ consist of
 morphisms of sheaves on each $X_{k}$, $\Sh(X_{\bullet})$ has naturally
 the structure of an abelian category (sums of morphisms, kernels and
 cokernels are simply taken space-wise).  Moreover, $\Sh(X_{\bullet})$
 has enough injectives, since simplicial sheaves on sites do so
 \cite[Prop.\ II.1.1, 2\superscript{nd} proof]{Milne80Etale-cohomology},
 \cite[p.\ 36]{Conrad03Cohomological-Descent}.
\end{remark}

\begin{definition}
 (\cite[5.1.13.1]{Deligne74Theorie-de-Hodge.-III}) The
 \emph{section functor} is the functor
 \begin{equation*}
  \Gamma\from \Sh(X_{\bullet})\to \cat{Ab},\quad
  F^{\bullet}\mapsto \ker(D^{1}_{0}-D^{1}_{1}),
 \end{equation*}
 where $D^{1}_{i}$ denotes the homomorphism of the groups of global
 sections $\Gamma(E^{0})\to \Gamma(E^{1})$, induced from the morphisms
 of sheaves $D^{1}_{i}\from {d_{1}^{i}}^{*}E^{0}\to E^{1}$.
\end{definition}

\begin{lemma}
 The functor $\Gamma$ is left exact.
\end{lemma}

\begin{definition}
 (\cite[5.2.2]{Deligne74Theorie-de-Hodge.-III}) The cohomology groups
 $H^{n}(X_{\bullet},E^{\bullet})$ are the right derived functors of
 the section functor $\Gamma$.
\end{definition}

\begin{tabsection}
 Since injective (or acyclic) resolutions in $\Sh(X_{\bullet})$ are not
 easily dealt with (cf.\ \cite[p.\ 36]{Conrad03Cohomological-Descent} or
 the explicit construction in \cite[Prop.\
 2.2]{Friedlander82Etale-homotopy-of-simplicial-schemes}), the groups
 $H^{n}(X_{\bullet},E^{\bullet})$ are notoriously hard to access.
 However, the following proposition provides an important link to 
 cohomology groups of the sheaves on each single space of $X_{\bullet}$.
\end{tabsection}

\begin{proposition}
 (\cite[5.2.3.2]{Deligne74Theorie-de-Hodge.-III}, \cite[Prop.\
 2.4]{Friedlander82Etale-homotopy-of-simplicial-schemes}) If
 $E^{\bullet}$ is a sheaf on $X_{\bullet}$, then there is a first
 quadrant spectral sequence with
 \begin{equation}\label{eqn:spectral_sequence}
  E_{1}^{p,q}=H^{q}_{\Sh}(X_{p},E^{p})\Rightarrow  
  H^{p+q}(X_{\bullet},E^{\bullet}).
 \end{equation}
\end{proposition}

\begin{remark}\label{rem:double_complex_for_spectral_sequence}
 We will need the crucial step from the proof of this proposition, so we repeat it here. It
 is the fact that the spectral sequence arises from a double complex
 \begin{equation*}
  \xymatrix@R=2em@C=2em{ &&&&\\
  F_{2}^{\bullet}\ar@{}[u]|(.75){\displaystyle\vdots} & \Gamma(F^{0}_{2}) \ar[r]^{d^{0}_{2}}\ar@{}[u]|(.75){\displaystyle\vdots}
  & \Gamma(F^{1}_{2})\ar[r]^{d^{1}_{2}}\ar@{}[u]|(.75){\displaystyle\vdots}	&	\Gamma(F^{2}_{2})	 \ar@{}[r]|(0.71){\displaystyle\cdots}\ar@{}[u]|(.75){\displaystyle\vdots}&\\%
  F_{1}^{\bullet} &	\Gamma(F^{0}_{1})\ar[u]	\ar[r]^{d^{0}_{1}}&	\Gamma(F^{1}_{1})\ar[u]\ar[r]^{d^{1}_{1}}
  & \Gamma(F^{2}_{1})\ar@{}[r]|(0.71){\displaystyle\cdots}\ar[u]	 &\\%
  F_{0}^{\bullet} &	\Gamma(F_{0}^{0})\ar[u]\ar[r]^{d^{0}}	& \Gamma(F_{0}^{1})\ar[u]\ar[r]^{d^{1}}	& \Gamma(F_{0}^{2})\ar[u]\ar@{}[r]|(0.71){\displaystyle\cdots}&	\\ 	
  & X_{0}		& X_{1}	& X_{2}\ar@{}[r]|(0.71){\displaystyle\cdots}&
  \ar "3,1" -/d 1em/; "2,1" -/u 1em/
  \ar "4,1" -/d 1em/; "3,1" -/u 1em/
  \ar@<.45em>   "5,4" -/r 1.25em/; "5,3" -/l 1.25em/
  \ar           "5,4" -/r 1.25em/; "5,3" -/l 1.25em/
  \ar@<-.45em>  "5,4" -/r 1.25em/; "5,3" -/l 1.25em/
  \ar@<.225em>  "5,3" -/r 1.25em/; "5,2" -/l 1.25em/
  \ar@<-.225em> "5,3" -/r 1.25em/; "5,2" -/l 1.25em/
  \ar "5,1" -/d:a(-50) 2em/; "5,5" -/d:a(-50) 2em/
  \ar "5,1" -/d:a(-50) 2em/ ; "1,1" -/d:a(-50) 2em/
  }
 \end{equation*}
 where each $F^{\bullet}_{{q}}$ is a sheaf on $X_{\bullet}$,
 $E^{q}\to F^{q}_{\bullet}$ is an injective resolution in $\Sh(X_{q})$
 \cite[Lemma 6.4]{Conrad03Cohomological-Descent} and $d^{p}_{q}$ is the
 alternating sum of morphisms induced from the $D_{i}^{p}$, respectively
 for each sheaf $F^{\bullet}_{q}$. Now taking the vertical differential
 first gives the above form of the $E_{1}$-term of the spectral sequence.
\end{remark}

\begin{corollary}\label{cor:acyclic_sheaves_give_augmentation_row_cohomology}
 If $E^{\bullet}$ is a sheaf on $X_{\bullet}$ such that each $E^{k}$ is
 acyclic on $X_{k}$, then $H^{n}(X_{\bullet},E^{\bullet})$ is the
 cohomology of the Moore complex of the cosimplicial group of sections
 of $E^{\bullet}$. More precisely, it is the cohomology of the complex
 $(\Gamma(X_{k},E^{k}),d)$ with differential given by
 \begin{equation*}
  d^{k}\gamma=\sum_{i=0}^{k}(-1)^{i}D^{k}_{i}{d_{k}^{i}}^{*}(\gamma)
  \quad\text{ for }\quad \gamma\in \Gamma(X_{k},E^{k}).
 \end{equation*}
\end{corollary}

\begin{proof}
 The $E_{1}$-term of the spectral sequence from the previous proposition
 is concentrated in the first column due to the acyclicity of $E^{k}$
 and yields the described cochain complex.
\end{proof}

\begin{remark}\label{rem:simplicail_sheaf_of_continuous_functions}
 The simplicial space that we will work with is the classifying
 space\footnote{The geometric realization of $BG_{\bullet}$ yields a
 model for the (topological) classifying space of $G$
 \cite{Segal68Classifying-spaces-and-spectral-sequences}, whence the
 name.} $BG_{\bullet}$ associated to $G$. It is given by setting
 $BG_{n}:=G^{\ktimes[n]}$ for $n\geq 1$ and $BG_{0}=\op{pt}$, and the
 standard simplicial maps are given by multiplying adjacent elements
 (respectively dropping the outermost off) and inserting identities.

 On $BG_{\bullet}$ we consider the sheaf $A_{\globc}^{\bullet}$, given
 on $BG_{n}=G^{n}$ as the sheaf of continuous $A$-valued functions
 $\underline{A}^{\cont}_{G^{n}}$. We turn this into a sheaf on
 $BG_{\bullet}$ by introducing the following morphisms $D_{i}^{n}$ and
 $S_{i}^{n}$. The structure maps on $BG_{\bullet}$ are in this case
 given by inclusions and projections. Indeed, the face maps factor
 through projections
 \begin{equation*}
  \xymatrix{G^{n}\ar@/_2em/[rr]^{d_{n}^{i}}\ar[r]^(.35){\cong}&G^{n-1}
  \ktimes G\ar[r]^(.55){\pr}&G^{n-1}}.
 \end{equation*}
 Thus ${d_{n}^{i}}^{*}\underline{A}^{\cont}_{G^{n-1}}(U)=C(d_{n}^{i}(U),A)$ and we may
 set
 \begin{equation*}
  (D_{i}^{n}f)(g_{0},\ldots,g_{n})=\begin{cases}
  f(d_{n}^{i}(g_{0},\ldots,g_{n}))&\text{if }i>0\\
  g_{0}.f(g_{1},\ldots,g_{n})&\text{if }i=0
  \end{cases}.
 \end{equation*}
 Similarly,
 \begin{equation*}
  {s_{n}^{i}}^{*}\underline{A}^{\cont}_{G^{n+1}}(U)=\lim_{\overrightarrow{~V~}}C(V,A),
 \end{equation*}
 where $V$ ranges through all open neighborhoods of $s_{n}^{i}(U)$, has
 a natural homomorphism $S_{i}^{n}$ to $\underline{A}^{\cont}_{G^{n}}(U)= C(U,A)$,
 given by precomposition with $s_{n}^{i}$.

 If, in addition, $G$ is a Lie group, then we also consider the slightly
 different simplicial space $BG^{\infty}_{\bullet}$ with
 $BG^{\infty}_{n}=G^{\mtimes[n]}$ and the same maps. If $A$ is a smooth
 $G$-module, we obtain in the same way the sheaf $A^{\bullet}_{\globs}$
 on $BG^{\infty}_{\bullet}$ by considering on each $BG^{\infty}_{n}$ the
 sheaf $\underline{A}^{\sm}_{G^{n}}$ of smooth $A$-valued functions (in
 order to make sense out of the latter we have to consider
 $BG^{\infty}_{\bullet}$ instead of $BG_{\bullet}$).
\end{remark}

\begin{definition}
 The \emph{continuous simplicial group cohomology} of $G$ with
 coefficients in $A$ is defined to be
 $H^{n}_{\simpc}(G,A):=H^{n}(BG_{\bullet},A^{\bullet}_{\globc})$. If $G$
 is a Lie group and $A$ a smooth $G$-module, then the \emph{smooth
 simplicial group cohomology} of $G$ with coefficients in $A$ is defined
 to be
 $H^{n}_{\simps}(G,A):=H^{n}(BG^{\infty}_{\bullet},A^{\bullet}_{\globs})$.
\end{definition}

\begin{lemma}\label{eqn:long_exact_coefficient_sequence_in_simplicial_cohom}
 If $A\xrightarrow{\alpha}B\xrightarrow{\beta}C$ is a short exact
 sequence of $G$-modules in $\cghaus$, then composition with
 $\alpha$ and $\beta$ induces a long exact sequence
 \begin{equation*}
  \cdots\to H^{n-1}_{\simpc}(G,C) \to H^{n}_{\simpc}(G,A)\to H^{n}_{\simpc}(G,B) \to H^{n}_{\simpc}(G,C) \to H^{n+1}_{\simpc}(G,A) \to\cdots.
 \end{equation*}
 If, moreover $G$ is a Lie group and
 $A\xrightarrow{\alpha}B\xrightarrow{\beta}C$ is a short exact sequence
 of smooth $G$-modules, then $\alpha$ and $\beta$
 induce a long exact sequence
 \begin{equation*}
  \cdots\to H^{n-1}_{\simps}(G,C) \to H^{n}_{\simps}(G,A)\to H^{n}_{\simps}(G,B) \to H^{n}_{\simps}(G,C) \to H^{n+1}_{\simps}(G,A) \to\cdots
 \end{equation*}
\end{lemma}

\begin{proof}
 Since kernels and cokernels of a sheaf $E^{\bullet}$ are simply the
 kernels and cokernels of $E^{k}$, this follows from the exactness
 of the sequences of sheaves of continuous functions
 $\underline{A}^{\cont}\to \underline{B}^{\cont}\to \underline{C}^{\cont}$ (and
 similarly for the smooth case).
\end{proof}

\begin{proposition}\label{prop:simp=glob_for_contractible_coefficients}
 If $G^{\ktimes[n]}$ is paracompact for each
 $n\geq 1$ and $A$ is contractible, then
 \begin{equation*}
  H^{n}_{\simpc}(G,A)\cong H^{n}_{\globc}(G,A).
 \end{equation*}
 If, moreover, $G$ is a Lie group, $A$ is a smoothly
 contractible\footnote{By this we mean that there exists a contraction
 of $A$ which is smooth as a map $[0,1]\mtimes A\to A$} smooth
 $G$-module and if $G^{\mtimes[n]}$ is smoothly
 paracompact for each $n\geq 1$, then
 \begin{equation*}
  H^{n}_{\simps}(G,A)\cong H^{n}_{\globs}(G,A).
 \end{equation*}
\end{proposition}

\begin{proof}
 In the case of contractible $A$ the sheaves $\underline{A}$ are soft
 and thus acyclic on paracompact spaces \cite[Thm.\
 II.9.11]{Bredon97Sheaf-theory}. The first claim thus follows from
 Corollary \ref{cor:acyclic_sheaves_give_augmentation_row_cohomology}.
 In the smooth case, the requirements are necessary to have the softness
 of the sheaf of smooth $A$-valued functions on each $G^{k}$ as well,
 since we then can extend sections from closed subsets (cf.\
 \eqref{eqn:sections_on_closed_subsets}) by making use of smooth
 partitions of unity.
\end{proof}

\begin{remark}
 The requirement on $G^{\ktimes[n]}$ to be paracompact for each
 $n\geq 1$ is for instance fulfilled if $G$ is metrizable, since then
 $G^{\ktimes[n]}=G^{\ptimes[n]}$ is so and metrizable spaces are
 paracompact. If $G$ is, in addition, a smoothly paracompact Lie group,
 then \cite[Cor.\
 16.17]{KrieglMichor97The-Convenient-Setting-of-Global-Analysis} shows
 that $G^{\mtimes[n]}$ is also smoothly paracompact.

 However, metrizable topological groups are not the most general
 compactly generated topological groups that can be of interest. Any
 $G$ that is a CW-complex has the property that $G^{\ktimes[n]}$
 is a CW-complex and thus is in particular paracompact.
\end{remark}

We now introduce a second important sheaf on $BG_{\bullet}$.

\begin{remark}
 For an arbitrary pointed topological space $(X,x)$ and an abelian
 topological group $A$, we denote by $A^{\locc}_{X}$ the sheaf
 \begin{equation*}
  U\mapsto \begin{cases}
  C_{\loc}(U,A)
  &\text{if }x\in U\\
  \Map(U,A) 	&\text{if }x\notin U
  \end{cases}
 \end{equation*}
 and call it the \emph{locally continuous sheaf} on $X$. If $X$ is a
 manifold and $A$ an abelian Lie group, then we similarly set
 \begin{equation*}
  A^{\locs}_{X}(U)=	\begin{cases}
  C_{\loc}^{\infty}(U,A)
  &\text{if }x\in U\\
  \Map(U,A) 	&\text{if }x\notin U.
  \end{cases}
 \end{equation*}
 Obviously, these sheaves have the sheaves of continuous functions
 $\underline{A}$ and of smooth functions $\underline{A}^{\sm}$ as sub sheaves.

 As in Remark \ref{rem:simplicail_sheaf_of_continuous_functions}, the
 sheaves $A^{\locc}_{G^{k}}$ assemble into a sheaf $A^{\bullet}_{\locc}$
 on $BG_{\bullet}$. Likewise, if $G$ is a Lie group and $A$ is smooth,
 the sheaves $A^{\locs}_{G^{k}}$ assemble into a sheaf
 $A^{\bullet}_{\locs}$ on $BG^{\infty}_{\bullet}$.
\end{remark}

We learned the importance of the following fact from
\cite{Pries09Smooth-group-cohomology}.

\begin{proposition}
 If $X$ is regular, then $A^{\locc}_{X}$ and $A^{\locs}_{X}$ are soft
 sheaves. In particular, these both sheaves are acyclic if $X$ is paracompact.
\end{proposition}

\begin{proof}
 In order to show that $A^{\locc}_{X}$ is soft we have to show that
 sections extend from closed subsets. Let $C\se X$ be closed and
 \begin{equation}\label{eqn:sections_on_closed_subsets}
  [f]\in A^{\locc}_{X}(C)=\lim_{\overrightarrow{~U~}}A^{\locc}_{X}(U)
 \end{equation}
 be a section over $C$, where the limit runs over all open
 neighborhoods of $C$ (cf.\ \cite[Th.\ II.9.5]{Bredon97Sheaf-theory}).
 Thus $[f]$ is represented by some $f\from U\to A$ for $U$ an open
 neighborhood of $C$. The argument now distinguishes the relative position 
 of the base point $x$ which enters the definition of $A^{\locc}_{X}$ with 
 respect to $U$. 

 If $x\in U$, then we may extend $f$ arbitrarily
 to obtain a section on $X$ which restricts to $[f]$. If $x\notin U$,
 then we choose $V\se X$ open with $C\se V$ and $x\notin \ol{V}$ and
 define $\wt{f}$ to coincide with $f$ on $U\cap V$ and to vanish
 elsewhere. This defines a section on $X$ restricting to $[f]$. This
 argument works for $A^{\locs}_{X}$ as well. Since soft sheaves on
 paracompact spaces are acyclic \cite[Thm.\
 II.9.11]{Bredon97Sheaf-theory}, this finishes the proof.
\end{proof}

Together with Corollary
\ref{cor:acyclic_sheaves_give_augmentation_row_cohomology}, this now implies

\begin{corollary}
If $G^{\ktimes[n]}$ is paracompact for all $n\geq 1$, then
 \begin{equation}\label{eqn:locally_continuous_cohomology_as_simplicial_cohomology}
  H^{n}(BG_{\bullet},A^{\bullet}_{\locc})\cong H^{n}_{\locc}(G,A).
 \end{equation}
 If $G$ is a Lie group and $G^{\ptimes[n]}$ is paracompact for all $n\geq 1$, then
 \begin{equation*}
  H^{n}(BG_{\bullet},A^{\bullet}_{\locs})\cong H^{n}_{\locs}(G,A).
 \end{equation*}
\end{corollary}

Note that the second of the previous assertions does not require each $G^{\mtimes[n]}$ to be \emph{smoothly} paracompact, plain paracompactness of the underlying topological space suffices.

\begin{remark}\label{rem:comparison_homomorphisms_from_simpc_to_locc}
 From the isomorphisms
 \eqref{eqn:locally_continuous_cohomology_as_simplicial_cohomology} we
 also obtain natural morphisms
 \begin{equation*}
  H^{n}_{\simpc}(G,A)\to H^{n}_{\locc}(G,A)\quad\text{ and }\quad
  H^{n}_{\simps}(G,A)\to H^{n}_{\locs}(G,A),
 \end{equation*}
 induced from the morphisms of sheaves
 $A^{\bullet}_{\globc}\to A^{\bullet}_{\locc}$ and
 $A^{\bullet}_{\globs}\to A^{\bullet}_{\locs}$ on $BG_{\bullet}$
 and $BG_{\bullet}^{\infty}$.
\end{remark}

\section{\v{C}ech cohomology}

\begin{tabsection}
 In this section, we will explain how to compute the cohomology groups
 introduced in the previous section in terms of \v{C}ech cocycles. This
 will also serve as a first touching point to the locally continuous
 (respectively smooth) cohomology from the first section in degree 2.
 The proof that all these cohomology theories are isomorphic in all
 degrees (under some technical conditions) will have to wait until
 Section \ref{sect:Comparison_Theorem}.
\end{tabsection}

\begin{definition}
 Let $X_{\bullet}$ be a \emph{semi-simplicial space}, i.e., a collection
 of topological spaces $(X_{k})_{k\in\N_{0}}$, together with continuous
 face maps $d_{k}^{i}\from X_{k}\to X_{k-1}$ for $i=0,\ldots,k$ such
 that $d_{k-1}^{i}\op{\circ}d_{k}^{j}=d_{k-1}^{j-1}\op{\circ}d_{k}^{i}$
 if $i<j$. Then a \emph{semi-simplicial cover} (or simply a
 \emph{cover}) of $X_{\bullet}$ is a semi-simplicial space $\mc{U}_{\bullet}$,
 together with a morphism
 $f_{\bullet}\from \mc{U}_{\bullet}\to X_{\bullet}$ of semi-simplicial
 spaces such that
 \begin{equation*}
  \mc{U}_{k}=\coprod _{j\in J_{k}}U^{j}_{k}
 \end{equation*}
 for $(U_{k}^{j})_{j\in J_{k}}$ an open cover of $X_{k}$ and
 $\left.f_{k}\right|_{U_{k}^{j}}$ is the inclusion
 $U_{k}^{j}\hookrightarrow X_{k}$. The cover is called \emph{good} if
 each $(U_{k}^{j})_{j\in J_{k}}$ is a good cover, i.e., all
 intersections $U_{k}^{j_{0}}\cap\ldots\cap U_{k}^{j_{l}}$ are contractible.
\end{definition}

\begin{remark}\label{rem:simplicial_covers}
 It is easy to construct semi-simplicial covers from covers of the $X_k$. In
 particular, we can construct good covers in the case that each $X_{k}$ admits
 good covers, i.e., each cover has a refinement which is a good cover. Indeed,
 given an arbitrary cover $(U^{i})_{i\in I}$ of $X_{0}$, denote $I$ by $J_{0}$
 and the cover by $(U_{0}^{j})_{j\in J_{0}}$. We then obtain a cover of $X_{1}$
 by pulling the cover $(U_{0}^{j})_{j\in J_{0}}$ back along $d_{1}^{0}$,
 $d_{1}^{1}$, $d_{1}^{2}$ and take a common refinement
 $(U_{1}^{j})_{j\in J_{1}}$ of the three covers. By definition, $J_{1}$ comes
 equipped with maps $\varepsilon _{1}^{1,2,3}\from J_{1}\to J_{0}$ such that
 $d_{1}^{i}(U_{1}^{j})\se U_{0}^{\varepsilon _{1}^{i}(j)}$. We may thus define
 the face maps of
 \begin{equation*}
  \mc{U}_{1}:=\coprod _{j\in J_{1}}U^{j}_{1}
 \end{equation*}
 to coincide with $d_{1}^{i}$. In this way we then proceed to arbitrary $k$. In
 the case that each $X_{k}$ admits good covers, we may refine the cover on each
 $X_{k}$ before constructing the cover on $X_{k+1}$ and thus obtain a good
 cover of $X_{\bullet}$.
 
 The previous construction can be made more canonical in the case that
 $X_{\bullet}=BG_{\bullet}$ for a compact Lie group $G$. In this case, there
 exists a bi-invariant metric on $G$, and we set
 \begin{equation*}
  r_{0}:=\sup\{r>0\mid U^{e,r}\text{ is geodesically convex}\},
 \end{equation*}
 where $U^{g,r}$ denotes the open ball around $g\in G$ of radius $r>0$. Then
 $(U^{g,r_{0}})_{g\in G}$ is a good open cover of $G$. Now the triangle inequality
 shows that $U^{g_{1},r_{0}/2}\cdot U^{g_{2},r_{0}/2}=U^{g_{1}g_{2},r_{0}}$, which is
 obviously true for $g_{1}=g_{2}=e$ and thus for arbitrary $g_{1}$ and $g_{2}$
 by the bi-invariance of the metric. Thus
 $(U^{g_{1},r_{0}/2}\times U^{g_{2},r_{0}/2})_{(g_{1},g_{2}) \in G^{2}}$ gives a cover
 of $G^{2}$ compatible with the face maps $d_{1}^{i}\from G^{2}\to G$.
 Likewise,
 \begin{equation*}
  (U^{g_{1},r_{0}/2^{k}}\times \ldots\times U^{g_{k},r_{0}/2^{k}})_{(g_{1},\ldots,g_{k})\in G^{k}}
 \end{equation*}
 gives a cover of $G^{k}$ compatible with the face maps
 $d_{k}^{i}\from G^{k}\to G^{k-1}$. Since each cover of $G^{k}$ consists of
 geodesically convex open balls in the product metric, this consequently
 comprises a canonical good open cover of $BG_{\bullet}$.
\end{remark}

\begin{definition}\label{def:cech_cohomology}
 Let $\mc{U}_{\bullet}$ be a cover of the semi-simplicial space
 $X_{\bullet}$ and ${E}^{\bullet}$ be a sheaf on
 $X_{\bullet}$\footnote{Sheaves on semi-simplicial spaces are defined
 likewise by omitting the degeneracy morphisms.}. Then the
 \emph{\v{C}ech complex} associated to $\mc{U}_{\bullet}$ and
 $E^{\bullet}$ is the double complex
 \begin{equation*}
  \check{C}^{p,q}(\mc{U}_{\bullet},E^{\bullet}):=\prod_{i_{0},\ldots,i_{q}\in I_{p}}
  E^{p}(U_{i_{0},\ldots,i_{q}}),
 \end{equation*}
 where we set, as usual,
 $U_{i_{0},\ldots,i_{q}}:=U_{i_{0}}\cap\ldots\cap U_{i_{q}}$. The two
 differentials
 \begin{equation*}
  d_{h}:=\sum_{i=0}^{p}(-1)^{i+q}D^{p}_{i} \op{\circ} {d_{p}^{i}}^{*} \from \check{C}^{p,q}(\mc{U}_{\bullet},E^{\bullet})\to \check{C}^{p+1,q}(\mc{U}_{\bullet},E^{\bullet})\quad\text{ and }\quad
  d_{v}:=\check{\delta}\from \check{C}^{p,q}(\mc{U}_{\bullet},E^{\bullet})\to \check{C}^{p,q+1}(\mc{U}_{\bullet},E^{\bullet})
 \end{equation*}
 turn $\check{C}^{p,q}(\mc{U}_{\bullet},E^{\bullet})$ into a double
 complex. We denote by
 $\check{H}^{n}(\mc{U}_{\bullet},E^{\bullet})$ the cohomology of the
 associated total complex and call it the \emph{\v{C}ech Cohomology} of
 $E^{\bullet}$ with respect to $\mc{U}_{\bullet}$.
\end{definition}

\begin{proposition}
 Suppose $G^{\ktimes[n]}$ is paracompact for each $n\geq 1$ and that
 $\mc{U}_{\bullet}$ is a good cover of $BG_{\bullet}$\footnote{We may also
 interpret $BG_{\bullet}$ as a semi-simplicial space by forgetting the
 degeneracy maps.}. If $A\xrightarrow{\alpha}B\xrightarrow{\beta}C$ is a short
 exact sequence of $G$-modules in $\cghaus$, then composition with $\alpha$ and
 $\beta$ induces a long exact sequence
 \begin{equation*}
  \cdots\to  \check{H}^{n-1}(\mc{U}_{\bullet},C^{\bullet}_{\globc}) \to
  \check{H}^{n}(\mc{U}_{\bullet},A^{\bullet}_{\globc}) \to
  \check{H}^{n}(\mc{U}_{\bullet},B^{\bullet}_{\globc}) \to
  \check{H}^{n}(\mc{U}_{\bullet},C^{\bullet}_{\globc}) \to
  \check{H}^{n+1}(\mc{U}_{\bullet},A^{\bullet}_{\globc}) \to \cdots.
 \end{equation*}
 Moreover, for each sheaf $E^{\bullet}$ on $BG_{\bullet}$ there is a first
 quadrant spectral sequence with
 \begin{equation*}
  E_{1}^{p,q}\cong \check{H}^{q}(|G|^{\ktimes[p]},E^{p})\Rightarrow \check{H}^{p+q}(\mc{U}_{\bullet},E^{\bullet}).
 \end{equation*}
 In particular, if $A$ is contractible, then
 \begin{equation*}
  \check{H}^{n}(\mc{U}_{\bullet},A^{\bullet}_{\globc})\cong H^{n}_{\globc}(G,A).
 \end{equation*}
\end{proposition}

\begin{proof}
 Each short exact sequence $A\to B\to C$ induces a short exact sequence
 of the associated double complexes and thus a long exact sequence
 between the cohomologies of the total complexes. The columns of the
 double complex $\check{C}^{p,q}(\mc{U}_{\bullet},E^{\bullet})$ are just
 the \v{C}ech complexes of the sheaf $E^{p}$ on $G^{p}$ for the open
 cover $\mc{U}_{p}$. Since the latter is good by assumption, the
 cohomology of the columns is isomorphic to the \v{C}ech cohomology of
 $G^{p}$ with coefficients in the sheaf $\underline{A}$.

 If $A$ is contractible, then the sheaf $\underline{A}$ is soft on each
 $G^{\ktimes[n]}$ and thus acyclic. Hence the $E_{1}$-term of the
 spectral sequence is concentrated in the first column. Since
 $E_{1}^{0,q}=C(G^{q},A)$ and the horizontal differential is just the
 standard group differential, this shows the claim.
\end{proof}

\begin{remark}\label{rem:morphism_from_locally_continuous_to_Cech_cohomology}
 For a connected topological group $G$ and a topological $G$-module $A$
 we will now explain how to construct an isomorphism
 $H^{2}_{\loctop}(G,A)\cong \check{H}^{2}(\mc{U}_{\bullet},A_{\globc}^{\bullet})$
 in quite explicit terms (where $\mc{U}_{\bullet}$ now is a good cover
 of the semi simplicial space $(G^{\ptimes[n]})_{n\in\N_{0}}$). To a
 cocycle $f\in C_{\locc}(G\ptimes G,A)$ with $\dgp f=0$ we associate the
 group $A\times_{f}G$ with underlying set $A\times G$ and multiplication
 $(a,g)\cdot (b,h)=(a+g.b+f(g,h),gh)$. Assuming that $U\se G$ is such
 that $\left.f\right|_{U\times U}$ is continuous and $V\se U$ is an open
 identity neighborhood with $V=V^{-1}$ and $V^{2}\se U$, there exists a
 unique topology on $A\times_{f}G$ such that
 $A\times V\hookrightarrow A\times _{f}G$ is an open embedding. In
 particular, $\pr_{2}\from A\times_{f}G\to G$ is a continuous
 homomorphism and $x\mapsto (0,x)$ defines a continuous section thereof
 on $V$. Consequently, $A\times_{f}G\to G$ is a continuous principal
 $A$-bundle.

 The topological type of this principal bundle is classified by a
 \v{C}ech cocycle $\tau(f)$, which can be obtained from the system of
 continuous sections
 \begin{equation*}
  \sigma_{g}\from gV\to A\times_{f}G,\quad x\mapsto (0,g)\cdot
  \sigma(g^{-1}x)=(f(g,g^{-1}x),x),
 \end{equation*}
 the associated trivializations
 $A\times gV\ni (a,x)\mapsto \sigma_{g}(x)\cdot (a,e)=(f(g,g^{-1}x)+x.a,x)\in\pr_{2}^{-1}(gV)$
 and is thus given on the cover $(gV)_{g\in G}$ by
 \begin{equation*}
  \tau(f)_{g_{1},g_{2}}\from g_{1}^{~}V\cap g_{2}^{~}V\to A,
  \quad x\mapsto 
  f(g_{2}^{~},g_{2}^{-1}x)-f(g_{1}^{~},g_{1}^{-1}x)=g_{1}^{~}.f(g_{1}^{-1}g_{2}^{~},g_{2}^{-1}x)-f(g_{1}^{~},g_{1}^{-1}g_{2}^{~}).
 \end{equation*}
 The multiplication
 $\mu\from (A\times_{f}G)\times(A\times_{f}G)\to A\times_{f}G$ may be
 expressed in terms of these local trivializations (although it might
 not be a bundle map in the case of nontrivial coefficients). For this,
 we pull back the cover $(gV)_{g\in G}$ via the multiplication to
 $G\times G$ and take a common refinement of this with the cover
 $(gV\times hV)_{(g,h)\in G\times G}$, over which the bundle
 $(A\times_{f}G)\times(A\times_{f}G)\to G\times G$ trivializes. A direct
 verification shows that $(V_{g,h})_{(g,h)\in g\times G}$ with
 \begin{equation*}
  V_{g,h}:=\{(x,y)\in G\times G:x\in gV, y\in hV,xy\in ghV\}
 \end{equation*}
 and the obvious maps does the job. Expressing $\mu$ in terms of these
 local trivializations, we obtain the representation
 \begin{equation*}
  ((a,x),(b,y))\mapsto
  \big((xy)^{-1}.\big[f(g,g^{-1}x)+a.x+x.f(h,h^{-1}y)+xy.b+f(x,y)-f(gh,(gh)^{-1}xy)\big],xy\big)
 \end{equation*}
 for $(x,y)\in V_{(g,h)}$. Since this is a continuous map
 $A^{2}\times V_{(g,h)}\to A\times V_{gh}$ and since $G$ acts
 continuously on $A$ it follows that
 \begin{equation*}
  \mu(f)_{g,h}\from V_{g,h}\to A,\quad (x,y)\mapsto
  f(g,g^{-1}x)+x.f(h,h^{-1}y)+f(x,y)-f(gh,(gh)^{-1}xy)
 \end{equation*}
 is indeed a continuous map.
 A straightforward computation with the definitions of $d_{v}$, $d_{h}$
 from Definition \ref{def:cech_cohomology} and the definitions of
 $D^{k}_{i}$ from Remark
 \ref{rem:simplicail_sheaf_of_continuous_functions} shows that
 $d_{h}(\tau(f))=d_{v}(\mu(f))$ in this case. Moreover, the cocycle
 identity for $f$ shows that $d_{h}(\mu(f))=0$. Thus $(\mu(f),\tau(v))$
 comprise a cocycle in the total complex of
 $\check{C}^{p,q}(\mc{U}_{\bullet},E^{\bullet})$ if we extend
 $(gV)_{g\in G}$ and $(V_{g,h})_{(g,h\in G\times G)}$ to a cover of
 $BG_{\bullet}$ as described in Remark \ref{rem:simplicial_covers}.

 The reverse direction is more elementary. One associates to a cocycle
 $(\Phi,\tau)$ in the total complex of
 $\check{C}^{p,q}(\mc{U}_{\bullet},E^{\bullet})$ a principal bundle
 $A\to P_{\tau}\to G$ clutched from the \v{C}ech cocycle $\tau$. Then
 $\Phi$ defines a map $P_{\tau}\times P_{\tau}\to P_{\tau}$ (not
 necessarily a bundle map, if $G$ acts nontrivially on $A$) whose
 continuity and associativity may be checked directly in local
 coordinates. Thus $P_{\tau}\to G$ is an abelian extension given by an
 element in $H^{2}_{\locc}(G,A)$. By making the appropriate choices, one
 sees that these constructions are inverse to each other on the nose.
\end{remark}

\section{The Comparison Theorem via soft modules}
\label{sect:Comparison_Theorem}

\begin{tabsection}
 We now describe a method for deciding whether certain cohomology theoreis
 are isomorphic. The usual, and frequently used technique for this is to
 invoke Buchsbaum's criterion
 \cite{Buchsbaum55Exact-categories-and-duality}, which also runs under
 the name universal $\delta$-functor or ``satellites''
 \cite{CartanEilenberg56Homological-algebra,Grothendieck57Sur-quelques-points-dalgebre-homologique,Weibel94An-introduction-to-homological-algebra}.
 The point of this section is that a more natural requirement on the
 various cohomology groups, which can often be checked right away for
 different definitions, implies this criterion. The reader who is
 unfamiliar with these techniques might wish to consult the independent
 Section \ref{sect:universal_delta_functors} before continuing.

 In order to make the comparison accessible, we have to introduce yet
 another definition of cohomology groups $H^{n}_{\SM}(G,A)$ for a
 $G$-module $A$ in $\cghaus$ due to Segal and Mitchison
 \cite{Segal70Cohomology-of-topological-groups}. We give some detail on
 this in Section
 \ref{sect:some_information_on_moore_s_and_segal_s_cohomology_groups};
 for the moment it is only important to recall that
 $A\mapsto H^{n}_{\SM}(G,A)$ is a $\delta$-functor for exact sequences
 of locally contractible $G$-modules that are principal bundles
 \cite[Prop.\ 2.3]{Segal70Cohomology-of-topological-groups} and that for
 contractible $A$, one has natural isomorphisms
 $H^{n}_{\SM}(G,A)\cong H^{n}_{\globc}(G,A)$ \cite[Prop.\
 3.1]{Segal70Cohomology-of-topological-groups}.
\end{tabsection}

\begin{remark}
 In what follows, we will consider a special kind of classifying space
 functor, introduced by Segal in
 \cite{Segal68Classifying-spaces-and-spectral-sequences}. The
 classifying space $BG$ and the universal bundle $EG$ are constructed by
 taking $BG=|BG_{\bullet}|$ (where $|\mathinner{\cdot}|$ denotes the thin (or ordinary)
 geometric realization), and $EG=|EG_{\bullet}|$, where $EG_{\bullet}$
 denotes the simplicial space obtained from the nerve of the pair
 groupoid of $G$. The resulting $EG$ is contractible. The nice thing
 about this construction of $BG$ is that it is functorial and that the
 natural map $E(G\ktimes G)\to EG\ktimes EG$ is a homeomorphism. In
 particular, $EG$ and $BG$ are again abelian groups in $\cghaus$
 provided that $G$ is so.
\end{remark}

\begin{definition}\label{def:segalsCohomology}
 (cf.\ \cite{Segal70Cohomology-of-topological-groups}) On $\kC(G,A)$, we
 consider the $G$-action $(g.f)(x):=g.(f(g^{-1}\cdot x))$\footnote{This really
 is the action one wants to consider, as one sees in \cite[Prop.\
 3.1]{Segal70Cohomology-of-topological-groups}. Some calculations in \cite[Ex.\
 2.4]{Segal70Cohomology-of-topological-groups} seem to use the action
 $(g.f)(x)=f(g^{-1}\cdot x)$, we clarify this in Section
 \ref{sect:some_information_on_moore_s_and_segal_s_cohomology_groups}.}, which
 obviously turns $\kC(G,A)$ into a $G$-module in $\cghaus$. If $A$ is
 contractible, then we call the module $\kC(G,A)$ a \emph{soft module}.
 Moreover, for arbitrary $A$ we set $E_{G}(A):=\kC(G,EA)$\footnote{Note that
 $EA$ is still Hausdorff if $A$ is so, cf.\
 \cite{Seguins-Pazzis10The-geometric-realization-of-a-simplicial-Hausdorff-space-is-Hausdorff}.}
 and $B_{G}(A):=E_{G}(A)/i_{A}(A)$, where
 $i_{A}\from A\hookrightarrow \kC(G,EA)$ is the closed embedding
 $A\hookrightarrow EA$, composed with the closed embedding
 $EA\hookrightarrow \kC(G,EA)$ of constant functions.
\end{definition}

\begin{lemma}\label{lem:global_section_for_BGA}
 The sequence $A\to E_{G}(A)\to B_{G}(A)$ has a local continuous
 section. If $A$ is contractible, then it has a global continuous
 section.
\end{lemma}

\begin{proof}
 The first claim is contained in \cite[Prop.\
 2.1]{Segal70Cohomology-of-topological-groups}, the second follows from
 \cite[App.\ (B)]{Segal70Cohomology-of-topological-groups}.
\end{proof}

\begin{proposition}\label{prop:soft_modules_are_acyclic_for_globcont}
 Soft modules are acyclic for the globally continuous group
 cohomology, i.e., $H^{n}_{\globc}(G,\kC(G,A))$ vanishes for
 contractible $A$ and $n\geq 1$.
\end{proposition}

\begin{proof}
 This is already implicitly contained in \cite[Prop.\
 2.2]{Segal70Cohomology-of-topological-groups}. See also \cite[Prop.\
 17]{Pries09Smooth-group-cohomology} and Section
 \ref{sect:some_information_on_moore_s_and_segal_s_cohomology_groups}.
\end{proof}

\begin{tabsection}
 The following theorem now shows that all cohomology
 theories considered so far are in fact isomorphic, at least if the
 topology of $G$ is sufficiently well-behaved.
\end{tabsection}

\begin{theorem}[Comparison Theorem]\label{thm:comparison_theorem}
 Let $\cat{G-Mod}$ be the category of locally contractible $G$-modules
 in $\cghaus$. We call a sequence
 $A\xrightarrow{\alpha} B\xrightarrow{\beta} C$ in $\cat{G-Mod}$ short
 exact if the underlying exact sequence of abelian groups is short exact
 and $\alpha$ (or equivalently $\beta$) has a local continuous section.
 If $(H^{n}\from\cat{G-Mod}\to\cat{Ab})_{n\in\N_{0}}$ is a
 $\delta$-functor such that
 \begin{enumerate}
        \renewcommand{\labelenumi}{\theenumi}
        \renewcommand{\theenumi}{\arabic{enumi}.}
  \item \label{eqn:comparison_theorem1} $H^{0}(A)=A^{G}$ is the
        invariants functor
  \item \label{eqn:comparison_theorem3} $H^{n}(A)=H^{n}_{\globc}(G,A)$
        for contractible $A$,
 \end{enumerate}
 then $(H^{n})_{n\in\N_{0}}$ is equivalent to
 $(H^{n}_{\SM}(G,\mathinner{\:\cdot\:}))_{n\in\N_{0}}$ as
 $\delta$-functor. Moreover, each morphism between $\delta$-functors
 with properties \ref{eqn:comparison_theorem1} and
 \ref{eqn:comparison_theorem3} that is an isomorphism for $n=0$ is
 automatically an isomorphism of $\delta$-functors.
\end{theorem}

\begin{proof}
 The functors $I(A):=E_{G}(A)$ and $U(A):=B_{G}(A)$ make Theorem
 \ref{thm:moores_comparison_theorem} applicable. To check the
 requirements of the first part, we have to show that
 $H^{n\geq 1}(E_{G}(A))$ vanishes, which in turn follows from property
 \ref{eqn:comparison_theorem3} and Proposition
 \ref{prop:soft_modules_are_acyclic_for_globcont}.

 To check the requirements of the second part of Theorem
 \ref{thm:moores_comparison_theorem} we observe that if $f\from A\to B$
 is a closed embedding with a local continuous section, then $f(A)$ is
 also closed in $E_{G}(B)$ and thus we may set $Q_{f}:=E_{G}(B)/f(A)$.
 The local sections of $f\from A\to B$ and $B\to E_{G}(B)$ then also
 provide a section of the composition $A\to E_{G}(B)$, and
 $A\to E_{G}(B)\to Q_{f}$ is short exact. The morphism
 $B_{G}(A)\to Q_{f}$ can now be taken to be induced by
 $f_{*}\from E_{G}(A)\to E_{G}(B)$, since it maps $A$ to $f(A)$ by
 definition. Likewise, $\iota_{B}$ maps $f(A)\se B$ into
 $f(A)\se E_{G}(B)$, so induces a morphism
 $\gamma_{f}\from B/f(A)\cong C\to Q_{f}=E_{G}(B)/f(A)$. The diagrams
 \eqref{eqn:morphism_of_delta_functors} thus commute by construction.
\end{proof}

\begin{tabsection}
 The property of a $G$-module $A$ to be locally contractible is essential
 for providing a local section of the embedding $A\to E_{G}(A)$
 \cite[Prop.\ A.1]{Segal70Cohomology-of-topological-groups}. We will
 assume this from now on without any further reference.
\end{tabsection}

\begin{remark}
 Property \ref{eqn:comparison_theorem3} of the Comparison Theorem may be
 weakened to
 \begin{equation*}
  H^{n}(A)=H^{n}_{\globc}(G,A)
  \text{ for \emph{loop} contractible }A,
 \end{equation*}
 where loop contractible means that there exists a contracting homotopy
 $\rho \from [0,1]\times A\to A$ such that each $\rho _{t}$ is a group
 homomorphisms for each $t\in [0,1]$.
 
 If this is the case, then one may still apply Theorem
 \ref{thm:moores_comparison_theorem}: We first observe that the abelian group
 $EA$ is loop contractible. In fact, identifying $EA$ with the space of left
 continuous step functions on the unit interval as in \cite[Ex.\
 5.5.]{Fuchssteiner11Cohomology-of-local-cochains} and \cite[Rem.\ on p.\
 217]{BrownMorris77Embeddings-in-contractible-or-compact-objects} one gets an
 explicit function $\rho\from  [0,1]\times EA\to EA$ for which one directly
 sees that $\rho_{0}=*$, $\rho_{1}=\id_{A}$ and each single $\rho_{t}$ is a
 group homomorphism. Now it is important to observe that $\rho$ actually
 coincides with the contracting homotopy of $EA$ as constructed from
 \cite[Prop.\ 2.1]{Segal68Classifying-spaces-and-spectral-sequences}. Thus
 $\rho$ is also continuous and we may conclude that $EA$ is loop contractible,
 although the identification of $EA$ with the aforementioned space of step
 functions may not respect the topology in general. In particular,
 $E_{G}=\kC(G,EA)$ is loop contractible and thus $H^{n\geq 1}(E_{G}(A))$ still
 vanishes. In this case, it is then a consequence of Theorem
 \ref{thm:moores_comparison_theorem} that $H^{n}(A)=H^{n}_{\globc}(G,A)$ for
 \emph{all} contractible modules $A$.
\end{remark}

\begin{corollary}
 If $G^{\ktimes[n]}$ is paracompact for each $n\geq 1$, then
 $H^{n}_{\SM}(G,A)\cong H^{n}_{\simpc}(G,A)$.
\end{corollary}

\begin{corollary}
 If $G^{\ptimes[n]}$ is compactly generated for each $n\geq 1$, then we
 have $H^{n}_{\locc}(G,A)\cong H^{n}_{\SM}(G,A)$\footnote{This is also
 the main theorem in \cite{Pries09Smooth-group-cohomology}, whose proof
 remains unfortunately incomplete.}. If, moreover, each $G^{\ptimes[n]}$
 is paracompact, then the morphisms
 \begin{equation*}
  H^{n}_{\simpc}(G,A)\to H^{n}_{\locc}(G,A),
 \end{equation*}
 from Remark \ref{rem:comparison_homomorphisms_from_simpc_to_locc} are
 isomorphisms.
\end{corollary}

\begin{corollary}
 Let $G$ be a finite-dimensional Lie group, $\mathfrak{a}$ be a
 quasi-complete locally convex space on which $G$ acts smoothly,
 $\Gamma\se\mathfrak{a}$ be a discrete submodule and set
 $A=\mathfrak{a}/\Gamma$. Then the natural morphisms
 \begin{equation}\label{eqn:finite-dimensional_G}
  H^{n}_{\simps}(G,A)\to H^{n}_{\locs}(G,A)\to H^{n}_{\locc}(G,A)
  \leftarrow H^{n}_{\simpc}(G,A)
 \end{equation}
 are all isomorphisms.
\end{corollary}

\begin{proof}
 The second is an isomorphism by Proposition
 \ref{prop:locc=locs_in_finite_dimensions} and the third by the
 preceding corollary.
 Since $H^{n}_{\simps}(G,\Gamma)\to H^{n}_{\simpc}(G,\Gamma)$ is an
 isomorphism by definition and
 $H^{n}_{\simps}(G,\mf{a})\to H^{n}_{\simpc}(G,\mf{a})$ is an
 isomorphism by Proposition
 \ref{prop:simp=glob_for_contractible_coefficients} and
 \cite{HochschildMostow62Cohomology-of-Lie-groups}, the fist one in
 \eqref{eqn:finite-dimensional_G} is an isomorphism by the Five Lemma.
\end{proof}

\begin{corollary}
 If $G^{\ktimes[n]}$ is paracompact for each $n\geq 1$, and
 $\mc{U}_{\bullet}$ is a good cover of $BG_{\bullet}$, then
 $H^{n}_{\SM}(G,A)\cong\check{H}^{n}(\mc{U}_{\bullet},A^{\bullet}_{\globc})$.
\end{corollary}

\begin{remark}
 Analogously to Corollary
 \ref{cor:acyclic_sheaves_give_augmentation_row_cohomology} one sees
 that if each $G^{\ktimes[n]}$ is paracompact, $\mc{U}_{\bullet}$ is a
 good cover of $BG_{\bullet}$ and $E^{\bullet}$ is a sheaf on
 $BG_{\bullet}$ with each $E^{n}$ is acyclic, then
 $\check{H}^{n}(\mc{U}_{\bullet},E^{\bullet})$ is the cohomology of the
 first column of the $E_{1}$-term. This shows in particular that
 $\check{H}^{n}(\mc{U}_{\bullet},A^{\bullet}_{\locc})\cong H^{n}_{\locc}(G,A)$.
 Moreover, the morphism of sheaves
 $A^{\bullet}_{\globc}\to A^{\bullet}_{\locc}$ induces a morphism
 \begin{equation}
  \label{eqn:morphism_from_cech_to_locally_continuous_cohomology}
  \check{H}^{n}(\mc{U}_{\bullet},A^{\bullet}_{\globc})\to \check{H}^{n}(\mc{U}_{\bullet},A^{\bullet}_{\locc})
  \xrightarrow{\cong} H^{n}_{\locc}(G,A).
 \end{equation}
 This morphism can be constructed in (more or less) explicit terms by
 the standard staircase argument for double complexes with acyclic rows
 (note that by the acyclicity of $A^{n}_{\locc}$ we may choose for each
 locally smooth \v{C}ech $q$-cocycle
 $\gamma_{i_{0},\ldots,i_{q}}\from U_{i_{0}}\cap\ldots\cap U_{i_{q}}\to A$
 on $G^{p}$ a locally smooth \v{C}ech cochain
 $\eta_{i_{0},\ldots,i_{q-1}}$ such that $\check{\delta}(\eta)=\gamma$).
 It is obvious that
 \eqref{eqn:morphism_from_cech_to_locally_continuous_cohomology} defines
 a morphism of $\delta$-functors. From the previous results and the
 uniqueness assertion of Theorem \ref{thm:moores_comparison_theorem} it
 now follows that
 \eqref{eqn:morphism_from_cech_to_locally_continuous_cohomology} is in
 fact an isomorphism provided $G^{\ptimes[n]}$ is compactly generated
 and paracompact for each $n\geq 1$.
\end{remark}

\begin{remark}
 In \cite[Prop.\
 5.1]{Flach08Cohomology-of-topological-groups-with-applications-to-the-Weil-group}
 it is shown that for $G$ a topological group and $A$ a $G$-module, such
 that the sheaf of continuous functions has no cohomology, the
 cohomology group of
 \cite{Flach08Cohomology-of-topological-groups-with-applications-to-the-Weil-group}
 coincide with $H^{n}_{\globc}(G,A)$. By \cite[Lem.\
 6]{Flach08Cohomology-of-topological-groups-with-applications-to-the-Weil-group}
 we also have long exact sequences, so the cohomology groups from
 \cite[Sect.\
 3]{Flach08Cohomology-of-topological-groups-with-applications-to-the-Weil-group}
 (which are anyway very similar to $H^{n}_{\simpc}(G,A)$, see also
 \cite{Lichtenbaum09The-Weil-etale-topology-for-number-rings}) also
 agree with $H^{n}_{\SM}(G,A)$.

 There is a slight variation of the latter cohomology groups by
 Schreiber
 \cite{Schreiber11Differential-Cohomology-in-a-Cohesive-infty-Topos} in
 the smooth setting and over the big topos of all Cartesian smooth
 spaces. The advantage of this approach is that it is embedded in a
 general setting of differential cohomology. In the case that $G$ is
 compact and $A$ is discrete or $A=\mf{a}/\Gamma$ for $\mf{a}$
 finite-dimensional, $\Gamma\se\mf{a}$ discrete and $G$ acts trivially
 on $A$ it has been shown in \cite[Prop.\
 3.3.12]{Schreiber11Differential-Cohomology-in-a-Cohesive-infty-Topos}
 that the cohomology groups
 $H^{n}_{\op{Smooth\infty{}Grpd}}({{\mathbf{B}}}G,A)$ from
 \cite{Schreiber11Differential-Cohomology-in-a-Cohesive-infty-Topos} are
 isomorphic to\footnote{This assertion is not stated explicitly but
 follows from  \cite[Prop.\
 3.3.12]{Schreiber11Differential-Cohomology-in-a-Cohesive-infty-Topos}
 by the vanishing of $H^{n}_{\globs}(G,\mf{a})$ \cite[Thm.\
 1]{Est55On-the-algebraic-cohomology-concepts-in-Lie-groups.-I-II} and
 the long exact coefficient sequence.}
 $\check{H}^{n}(\mc{U}_{\bullet},A^{\bullet}_{\globs})$ (where
 $\mc{U}_{\bullet}$ is a good cover of $BG^{\infty}_{\bullet}$).
\end{remark}

\begin{remark}
 We now compare $H^{n}_{\locc}(G,A)$ with (one of) the cohomology groups from
 \cite{Moore76Group-extensions-and-cohomology-for-locally-compact-groups.-III}.
 For this we assume that $G$ is a second countable locally compact group of
 finite covering dimension. A Polish $G$-module is a separable complete
 metrizable\footnote{We will throughout assume that the metric is bounded. This
 is no lose of generality since we may replace each invariant metric $d(x,y)$
 with the topologically equivalent bounded invariant metric
 $\frac{d(x,y)}{1+d(x,y)}$.} abelian topological group $A$ together with a
 jointly continuous action $G\times A\to A$. Morphisms of Polish $G$-modules
 are continuous group homomorphisms intertwining the $G$-action. If $G$ is a
 locally compact group and $A$ is a Polish $G$-module, then
 $H^{n}_{\Moore}(G,A)$ denotes the cohomology of the cochain complex
 \begin{equation*}
  C^{n}_{\meas}(G,A):=\{f\from G^{n}\to A:f\text{ is Borel measurable}\}
 \end{equation*}
 with the group differential $\dgp$ from \eqref{eqn:group_differential}. It has
 already been remarked in
 \cite{Wigner73Algebraic-cohomology-of-topological-groups} that these are
 isomorphic to $H^{n}_{\simpc}(G,A)$, we give here a detailed proof of this and
 extend the result slightly.
 
 On the category of Polish $G$-modules we consider as short exact sequences
 those sequences $A\xrightarrow{\alpha} B\xrightarrow{\beta} C$ for which the
 underlying sequence of abstract abelian groups is exact, $\alpha$ is an
 (automatically closed) embedding and $\beta$ is open. From \cite[Prop.\
 11]{Moore76Group-extensions-and-cohomology-for-locally-compact-groups.-III} it
 follows that from this we obtain natural long exact sequences, i.e.,
 $H^{n}_{\Moore}(G,\:\mathinner{\cdot}\:)$ is a $\delta$-functor. Moreover, it
 follows from \cite[Prop.\
 3]{Wigner73Algebraic-cohomology-of-topological-groups} and from the remarks
 before \cite[Thm.\ 2]{Wigner73Algebraic-cohomology-of-topological-groups} that
 each locally continuous cochain $f\from G^{\ptimes[n]}\to C$ can be lifted to
 a locally continuous cochain $\wt{f}\from G^{\ptimes[n]}\to B$. This is due to
 the assumption on $G$ to be finite-dimensional. From this it follows as in
 Remark \ref{rem:long_exact_coefficient_sequence} that
 $A\xrightarrow{\alpha} B\xrightarrow{\beta} C$ also induces a long exact
 sequence for $H^{n}_{\locc}(G,\:\mathinner{\cdot}\:)$ (this is the reason for
 choosing $H^{n}_{\locc}(G,A)$ for this comparison).
 
 On the category of Polish $G$-modules we now consider the functors
 \begin{equation*}
  A \mapsto \ol{E}_{G}(A):=C(G,U(I,A)),
 \end{equation*}
 where $U(I,A)$ is the group of Borel functions from the unit interval $I$ to
 $A$ modulo those that vanish on a set of measure $0$. Moreover, $U(I,A)$ is a
 Polish $G$-module \cite[Sect.\
 2]{Moore76Group-extensions-and-cohomology-for-locally-compact-groups.-III} and
 coincides with the completion of the metric abelian topological group
 $A$-valued step-functions on the right-open unit interval $[0,1)$, endowed
 with the metric
 \begin{equation*}
  d(f,g ):=\int_{0}^{1}d_{A}(f(t)g(t))\;dt,
 \end{equation*}
 see also
 \cite{BrownMorris77Embeddings-in-contractible-or-compact-objects,Keesling73Topological-groups-whose-underlying-spaces-are-separable-Frechet-manifolds,HartmanMycielski58On-the-imbedding-of-topological-groups-into-connected-topological-groups}.
 In particular, $U(I,A)$ inherits the structure of a $G$-module and so does
 $\ol{E}_{G}(A)$. Moreover, it is contractible and thus $\ol{E}_{G}(A)$ is
 soft. Since $G$ is $\sigma$-compact we also have that $C(G,U(I,A))$ is
 completely metrizable.
 
 Now $A$ embeds as a closed submodule into $\ol{E}_{G}(A)$ and we set
 $\ol{B}_{G}(A):=\ol{E}_{G}(A)/A$. Thus
 \begin{equation*}
  A\to \ol{E}_{G}(A)\to \ol{B}_{G}(A)
 \end{equation*}
 becomes short exact since orbit projection of continuous group actions are
 automatically open. In virtue of Theorem \ref{thm:moores_comparison_theorem}
 and the fact that the locally continuous cohomology vanishes for soft modules
 this furnishes a morphism of $\delta$-functors from
 $H^{n}_{\locc}(G,\:\mathinner{\cdot}\:)$ to
 $H^{n}_{\Moore}(G,\:\mathinner{\cdot}\:)$ (the constructions of
 $Q_{f}, \beta_{f}$ and $\gamma_{f}$ from Theorem \ref{thm:comparison_theorem}
 carry over to the present setting). Moreover, the functors $A\mapsto I(A)$ and
 $A\mapsto U(A)$ that Moore constructs in \cite[Sect.\
 2]{Moore76Group-extensions-and-cohomology-for-locally-compact-groups.-III}
 satisfy $H^{n}_{\Moore}(I(A))=0$ \cite[Thm.\
 4]{Moore76Group-extensions-and-cohomology-for-locally-compact-groups.-III}.
 Thus Remark \ref{rem:weaker_comparison_theorem} shows that
 $H^{n}_{\locc}(G,\:\mathinner{\cdot}\:)$ and
 $H^{n}_{\Moore}(G,\:\mathinner{\cdot}\:)$ are isomorphic (even as
 $\delta$-functors) on the category of Polish $G$-modules. This also extends
 \cite[Thm.\ C]{AustinMoore10Continuity-properties-of-measurable-group-cohomology} to arbitrary
 locally contractible coefficients.
 
 In addition, this shows that the mixture of measurable and locally continuous
 cohomology groups $H^{n}_{lcm}(G,A)$ from
 \cite{KhedekarRajan12On-cohomology-theory-for-topological-groups} does
 coincide with $H^{n}_{\Moore}(G,A)$. Indeed, the morphism
 $H^{n}_{lcm}(G,A)\to H^{n}_{\Moore}(G,A)$ of $\delta$-functors \cite[Cor.\
 1]{KhedekarRajan12On-cohomology-theory-for-topological-groups} is surjective
 for each $n$ and contractible $A$ (since then
 $H^{n}_{\globc}(G,A)\to H^{n}_{\Moore}(G,A)$ is surjective) and also injective
 (since $H^{n}_{\globc}(G,A)\to H^{n}_{lcm}(G,A)\to H^{n}_{\locc}(G,A)$ is so).
 Thus $H^{n}_{lcm}(G,A)\cong H^{n}_{\Moore}(G,A)\cong H^{n}_{\globc}(G,A)$ for
 each $n$ and contractible $A$ and the Comparison Theorem shows that
 $H^{n}_{lcm}(G,\:\mathinner{\cdot}\:)$ is isomorphic to
 $H^{n}_{\locc}(G,\:\mathinner{\cdot}\:)$, also as $\delta$-functor.
\end{remark}

\begin{remark}\label{rem:bounded_cohomology}
 Whereas all preceding cohomology theories fit into the framework of the
 Comparison Theorem, bounded continuous cohomology
 \cite{Monod01Continuous-bounded-cohomology-of-locally-compact-groups,Monod06An-invitation-to-bounded-cohomology}
 does not. First of all, this concept considers locally compact $G$ and
 Banach space coefficients $A$, whence all of the above cohomology
 theories agree to give $H^{n}_{\globc}(G,A)$. The bounded continuous
 cohomology $H^{n}_{\bc}(G,A)$ is the cohomology of the sub complex
 of bounded continuous functions $(C_{\bc}(G^{n},A),\dgp)$. Thus there
 is a natural comparison map
 \begin{equation*}
  H^{n}_{\bc}(G,A)\to H^{n}_{\globc}(G,A)
 \end{equation*}
 which is obviously an isomorphism for compact $G$. However, bounded
 cohomology unfolds its strength not before considering non-compact
 groups, where the above map is in general not an isomorphism \cite[Ch.\
 9]{Monod01Continuous-bounded-cohomology-of-locally-compact-groups},
 even not for Lie groups \cite[Ex.\
 9.3.11]{Monod01Continuous-bounded-cohomology-of-locally-compact-groups}.
 In fact, bounded cohomology is \emph{designed} to make the above map
 \emph{not} into an isomorphism for measuring the deviation of $G$
 from being compact.
\end{remark}

\begin{tabsection}
 Despite the last example, the properties of the Comparison Theorem seem
 to be the essential ones for a large class of important concepts of
 cohomology groups for topological groups. We thus give it the following
 name.
\end{tabsection}

\begin{definition}
 A \emph{cohomology theory for $G$} is a $\delta$-functor
 $(F^{n}\from \cat{G-Mod}\to \cat{Ab})_{n\in\N}$ satisfying conditions
 \ref{eqn:comparison_theorem1} and \ref{eqn:comparison_theorem3} of the
 Comparison Theorem.
\end{definition}

\begin{remark}\label{rem:properties}
 We end this section with listing properties that any cohomology theory
 for $G$ has. We will always indicate the concrete model that we are
 using, the isomorphisms of the models are then due to the corollaries
 of this section. Parts of these facts have already been established for
 the various models in the respective references.
 \begin{enumerate}
  \item If $A$ is discrete and each $G^{\ktimes[n]}$ is paracompact,
        then
        $H^{n}_{\SM}(G,A)\cong H^{n}_{\pi_{1}(BG)}(BG,\underline{A})$ is
        the cohomology of the topological classifying space twisted by
        the $\pi_{1}(BG)\cong \pi_{0}(G)$-action on $A$ (note that
        $G_{0}$ acts trivially since $A$ is discrete). This follows from
        \cite[Prop.\ 3.3]{Segal70Cohomology-of-topological-groups}, cf.\
        also \cite[6.1.4.2]{Deligne74Theorie-de-Hodge.-III}. If,
        moreover, $G$ is $(n-1)$-connected, then
        $H^{n+1}_{\SM}(G,A)\cong \Hom(\pi_{n}(G),A)$.
  \item If $G$ is contractible and each $G^{\ptimes[n]}$ is compactly
        generated, then
        $H^{n}_{\SM}(G,A)\cong H^{n}_{\locc}(G,A)\cong H^{n}_{\globc}(G,A)$.
        This follows from \cite[Thm.\
        5.16]{Fuchssteiner11A-spectral-sequence-connecting-continuous-with-locally-continuous-group-cohomology}.
  \item If $G$ is compact and $A=\mf{a}/\Gamma$ for $\mf{a}$ a
        quasi-complete locally convex space which is a continuous
        $G$-module and $\Gamma$ a discrete submodule, then
        $H^{n}_{\SM}(G,A)\cong H^{n+1}_{\pi_{1}(BG)}(BG,\Gamma)$. This
        follows from the vanishing of
        $H^{n}_{\SM}(G,\mf{a})\cong H^{n}_{\globc}(G,\frak{a})$ (cf.\
        \cite[Thm.\ 2.8]{Hu52Cohomology-theory-in-topological-groups} or
        \cite[Lem.\
        IX.1.10]{BorelWallach00Continuous-cohomology-discrete-subgroups-and-representations-of-reductive-groups})
        and the long exact sequence induced from the short exact
        sequence $\Gamma\to \mf{a}\to A$. In particular, if $G$ is a
        compact Lie group and $A$ is finite-dimensional, then
        \begin{equation*}
         H^{n}_{\locc}(G,A)\cong H^{n}_{\locs}(G,A)\cong
         H^{n+1}_{\pi_{1}(BG)}(BG,\Gamma).
        \end{equation*}
 \end{enumerate}
\end{remark}

\section{Examples and applications} \label{sect:examples}

\begin{tabsection}
 The main motivation for this paper is that locally continuous and
 locally smooth cohomology are somewhat easy to handle, but lacked so
 far a conceptual framework. On the other hand, the simplicial
 cohomology groups or the ones introduced by Segal and Mitchison are
 hard to handle in degrees $\geq 3$. We will give some results that one
 can derive from the interaction of these different concepts.
\end{tabsection}

\begin{example}\label{ex:string_cocycles}
 There are several cocycles (or, more precisely,
 cohomology classes) which deserve to be named ``string cocycle'' (or, more
 precisely, ``string class''). For this example, we assume that $G$ is a
 compact simple and 1-connected Lie group (which is thus automatically
 2-connected). There exists for each $g\in G$ a path
 $\alpha_{g}\in C^{\infty}([0,1],G)$ with $\alpha_{g}(0)=e$,
 $\alpha_{g}(1)=g$ and for each $g,h\in G$ a filler\footnote{From the
 2-connectedness of $G$ it only follows that there exist continuous
 fillers, that these can chosen to be smooth follows from the density of
 $C^{\infty}(\Delta^{n},G)$ in $C(\Delta^{n},G)$ \cite[Cor.\
 14]{Wockel06A-Generalisation-of-Steenrods-Approximation-Theorem}.}
 $\beta_{g,h}\in C^{\infty}(\Delta^{2},G)$ for the triangle
 $(\dgp \alpha)(g,h)= g.\alpha_{h}-\alpha_{gh}+\alpha_{g}$ (Figure
 \ref{fig:pic_triangle}).
 \begin{figure}[htbp]
  \centering \includegraphics[width=0.5\textwidth]{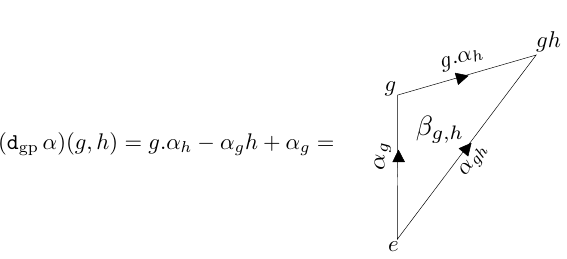}
  \caption{$\beta_{g,h}$ fills $(\dgp \alpha)(g,h)$}
  \label{fig:pic_triangle}
 \end{figure}
 
 \noindent Moreover,
 $(\dgp \beta)(g,h,k)=g.\beta_{h,k}-\beta_{gh,k}+\beta_{g,hk}-\beta_{g,h}$
 bounds a tetrahedron which can be filled with
 $\gamma_{g,h,k}\in C^{\infty}(\Delta^{3},G)$ (Figure
 \ref{fig:pic_cocycle}).
 \begin{figure}[htbp]
  \centering \includegraphics[width=0.65\textwidth]{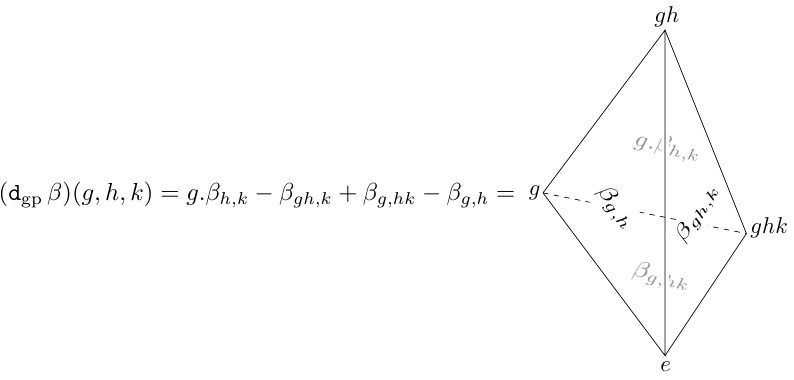}
  \caption{$\gamma_{g,h,k}$ fills $(\dgp \beta)(g,h,k)$}
  \label{fig:pic_cocycle}
 \end{figure}
 
 \noindent In addition, $\alpha$, $\beta$ and $\gamma$, interpreted as
 maps $G^{n}\to C^{\infty}(\Delta^{n},G)$ for $n=1,2,3$, can be chosen
 to be smooth on some identity neighborhood. From these choices we can
 now construct the following cohomology classes (which in turn are
 independent of the above choices as a straightforward check shows,
 cf.\ \cite[Rem.\
 1.12]{Wockel08Categorified-central-extensions-etale-Lie-2-groups-and-Lies-Third-Theorem-for-locally-exponential-Lie-algebras}).
 \begin{enumerate}
  \item Since
        $\partial \dgp(\gamma)=\dgp(\partial \gamma)=\dgp^{2}\beta =0$,
        the map
        \begin{equation*}
         (g,h,k,l)\mapsto(\dgp \gamma)(g,h,k,l)
        \end{equation*}
        takes values in the singular 3-cycles on $G$ and thus gives rise
        to map
        $\theta_{3}\from G^{4}\to H_{3}(G)\cong \pi_{3}(G)\cong \Z$ (see
        also Example \ref{ex:path-space_construction}). This map is
        locally smooth since $\gamma$ was assumed to be so and it is a
        cocycle since $\dgp (\dgp(\gamma))=0$ (note that it is not a
        coboundary since $\gamma$ does not take values in the singular
        cycles but only in the singular chains).
  \item The cocycle $\sigma_{3}\from G^{3}\to U(1)$ from \cite[Ex.\
        4.10]{Wockel08Categorified-central-extensions-etale-Lie-2-groups-and-Lies-Third-Theorem-for-locally-exponential-Lie-algebras}
        obtained by setting
        \begin{equation*}
         \sigma_{3}(g,h,k):=\exp\left(\int_{\gamma_{g,h,k}}\omega\right),
        \end{equation*}
        where $\omega$ is the left-invariant 3-from on $G$ with
        $\omega(e)=\langle [\mathinner{\cdot},\mathinner{\cdot}],\mathinner{\cdot}\rangle$
        normalized such that $[\omega]\in H^{3}_{\op{dR}}(G)$ gives a
        generator of $H^{3}_{\op{dR}}(G,\Z)\cong \Z$ and
        $\exp\from \R\to U(1)$ is the exponential function of $U(1)$
        with kernel $\Z$. Since $\omega$ is in particular an integral
        3-form, this implies that $\sigma_{3}$ is a cocycle because
        $\dgp(\gamma)(g,h,k,l)$ is a piece-wise smooth singular cycle
        and thus
        \begin{equation*}
         \dgp \sigma_{3}(g,h,k,l)=\exp\left(\int _{\dgp \gamma(g,h,k,l)}\omega \right)=1.
        \end{equation*}
        Since $\gamma$ is smooth on some identity neighborhood,
        $\sigma_{3}$ is so as well. Now
        \begin{equation*}
         \wt{\sigma}_{3}(g,h,k):=\int_{\gamma(g,h,k)}\omega
        \end{equation*}
        provides a locally smooth lift of $\sigma_{3}$ to $\R$. Thus the
        homomorphism
        $\delta\from H^{3}_{\locs}(G,U(1))\to H^{4}_{\locs}(G,\Z)$ maps
        $[\sigma_{3}]$ to $[\theta_{3}]$ since
        \begin{equation*}
         \dgp \wt{\sigma}_{3}=\int_{\dgp\gamma}\omega
        \end{equation*}
        and integration of piece-wise smooth representatives along
        $\omega$ provides the isomorphism $\pi_{3}(G)\cong \Z$. We will
        justify calling $\sigma_{3}$ a string cocycle in Remark
        \ref{rem:string_cocycle_2}.
  \item The locally smooth cocycles arising as characteristic cocycles
        \cite[Lem.\
        3.6.]{Neeb07Non-abelian-extensions-of-infinite-dimensional-Lie-groups}
        from the strict models
        \cite{BaezCransStevensonSchreiber07From-loop-groups-to-2-groups,NikolausSachseWockel13A-Smooth-Model-for-the-String-Group}
        of the string 2-group. In the case of the model from
        \cite{BaezCransStevensonSchreiber07From-loop-groups-to-2-groups}
        this gives precisely $\sigma_{3}$.
 \end{enumerate}
 Suppose $\mc{U}_{\bullet}$ is a good cover of $BG_{\bullet}$. The model
 from
 \cite{Schommer-Pries10Central-Extensions-of-Smooth-2-Groups-and-a-Finite-Dimensional-String-2-Group}
 is constructed by showing that
 $\check{H}^{3}(\mc{U}_{\bullet},U(1)^{\bullet}_{\globs})$ classifies
 central extensions of finite-dimensional group stacks
 \begin{equation*}
  [{*}/U(1)]\to [\Gamma] \to G
 \end{equation*}
 and then taking the isomorphism
 \begin{equation*}
  \check{H}^{3}(\mc{U}_{\bullet},U(1)^{\bullet}_{\globs})\cong H^{3}_{\SM}(G,U(1))\cong H^{4}(BG,\Z)\cong \Z
 \end{equation*}
 (cf.\ Remark \ref{rem:properties}), yielding for each generator a model
 for the string group. We will see below that the classes from above are
 also generators in the respective cohomology groups and thus represent
 the various properties of the string group. For instance, we expect
 that the class $[\sigma_{3}]$ will be the characteristic class for
 representations of the string group.
\end{example}

The previous construction can be generalized as follows.

\begin{example}\label{ex:path-space_construction}
 Let $G$ be a $(n-1)$-connected Lie group and denote by
 $C^{\infty}_{*}(\Delta^{k},G)$ the group of based smooth $k$-simplices
 in $G$ (the same construction works for locally contractible
 topological groups and the continuous $k$-simplices). Then we may
 choose for each $1\leq k\leq n$ maps
 \begin{equation*}
  \alpha_{k}\from G^{k}\to C^{\infty}_{*}(\Delta^{k},G),
 \end{equation*}
 such that each $\alpha_{k}$ is smooth on some identity neighborhood and
 that
 \begin{equation*}
  \partial\alpha_{k}(g_{1},\ldots,g_{k})=\dgp(\alpha_{k-1})(g_{1},\ldots,g_{k}).
 \end{equation*}
 In the latter formula, we interpret $C^{\infty}_{*}(\Delta^{k},G)$ as a
 subset of the group $\langle C(\Delta^{k},G)\rangle_{\Z}$ of singular
 $k$-chains in $G$, which becomes a $G$-module if we let $G$ act by left
 multiplication. Since $G$ is $(n-1)$-connected, we can inductively
 choose $\alpha_{k}$, starting with $\alpha_{0}\equiv e$.

 Now consider the map
 \begin{equation*}
  \theta_{n}:=\dgp(\alpha_{n})\from G^{n+1}\to \langle C(\Delta^{n},G)\rangle_{\Z}.
 \end{equation*}
 Since
 \begin{equation}
  \partial \theta_{n}=\partial \dgp(\alpha_{n})=\dgp(\partial \alpha_{n})=\dgp^{2}(\alpha_{n-1})=0,
 \end{equation}
 $\theta_{n}$ takes values in the singular $n$-cycles on $G$ and thus
 gives rise to a map
 $\theta_{n}\from G^{n+1}\to H_{n}(G)\cong \pi_{n}(G)$.
 Moreover, $\theta_{n}$ is a group cocycle %
 and it is locally smooth since $\alpha_{n}$ is so. Of course, this
 means here that $\theta_{n}$ even vanishes on some identity
 neighborhood (in the product topology). It is straightforward to show
 that different choices for $\alpha_{k}$ yield equivalent cocycles.

 These are the characteristic cocycles for the $n$-fold extension
 \begin{equation}\label{eqn:n-fold_extension}
  \pi_{n}(G)\to \wt{\Omega ^{n}G}\to P_{e}\Omega^{n-1}G\to \cdots\to P_{e}\Omega G\to P_{e}G \to G,
 \end{equation}
 ($P_{e}$ denoted pointed paths and $\Omega $ pointed loops) of
 topological groups spliced together from the short exact sequences
 \begin{equation*}
  \pi_{n}(G)\to \wt{\Omega^{n}G}\to \Omega^{n}G
  \quad\text{ and }\quad
  \Omega^{n}G\to P_{e}\Omega^{n-1}G \to \Omega^{n-1}G\text{ for }n\geq 0.
 \end{equation*}
 Moreover, the exact sequence
 \begin{equation*}
  \wt{\Omega ^{n}G}\to \Omega^{n-1}G\to \cdots\to \Omega G\to P_{e}G
 \end{equation*}
 gives rise to a simplicial topological group $\Pi_{n}(G)$ and we have
 canonical morphisms
 \begin{equation*}
  B^{n}\pi_{n}(G)\to \Pi_{n}(G)\to \ul{G}.
 \end{equation*}
 Here, $B^{n}\pi_{n}(G)$ is the nerve of the $(n-1)$-groupoid with only
 trivial morphisms up to $(n-2)$ and $\pi_{n}(G)$ as $(n-1)$-morphisms
 and $\ul{G}$ is the nerve of the groupoid with objects $G$ and only
 identity morphisms. Taking the geometric realization
 $|\mathinner{\cdot}|$ gives (at least for metrizable $G$) now an
 extension of groups in $\cghaus$
 \begin{equation*}
  K(n,\pi_{n}(G))\simeq |B^{n}\pi_{n}(G)|\to |\Pi_{n}(G)|\to |G|=G,
 \end{equation*}
 which can be shown to be an $n$-connected cover
 $G\langle n\rangle\to G$ with the same methods as in
 \cite{BaezCransStevensonSchreiber07From-loop-groups-to-2-groups}.
\end{example}

\begin{remark}\label{rem:X-Mod}
 Recall that a crossed module $\mu:M\to N$ is a group homomorphism
 together with an action by automorphisms of $N$ on $M$ such that $\mu$
 is equivariant and such that for all $m,m'\in M$, the Peiffer identity
 \begin{equation*}
  \mu(m). m'=mm'm^{-1} 
 \end{equation*}
 holds. Taking into account topology, we suppose that $M$ and $N$ are
 groups in $\cghaus$, $\mu$ is continuous and $(n,m)\mapsto n.m$ is
 continuous. We call a closed subgroup $H$ of a group in $\cghaus$
 {\it split} if the multiplication map $G\ktimes H\to G$ defines a
 topological $H$-principal bundle (see \cite[Def.\
 2.1]{Neeb07Non-abelian-extensions-of-infinite-dimensional-Lie-groups}).
 We will throughout use the constructions in the smooth setting from
 \cite{Neeb07Non-abelian-extensions-of-infinite-dimensional-Lie-groups},
 which carry over to the present topological setting. In this case, we
 have in particular that $G\to G/H$ has a continuous local section. To
 avoid pathological cases%
, we suppose that all our crossed modules are {\it
 topologically split}, i.e., we suppose that $\ker(\mu)$ is a split
 topological subgroup of $M$, that $\im(\mu)$ is a split topological
 subgroup of $N$, and that $\mu$ induces a homeomorphism
 $M/\ker(\mu)\cong\im(\mu)$. 
\end{remark}

Using the above methods, we can now show the following:

\begin{theorem}
 If each $G^{\ptimes[n]}$ is compactly generated, then the set of
 equivalence classes of crossed modules with cokernel $G$ and kernel $A$
 is in bijection with $H^3_{\locc}(G,A)$.
\end{theorem}

\begin{proof}
 It is standard to associate to a (topologically split) crossed module a
 locally continuous $3$-cocycle (see \cite[Lem.\
 3.6]{Neeb07Non-abelian-extensions-of-infinite-dimensional-Lie-groups}).
 To show that this defines an injection of the set of equivalence
 classes into $H^3_{\locc}(G,A)$, we use the continuous version of
 \cite[Th.\
 2.17]{Neeb07Non-abelian-extensions-of-infinite-dimensional-Lie-groups}.
 Namely, if $A\to M\to N\to G$ is sent to the trivial class, 
 the existence of an extension $M\to\hat{G}\xrightarrow{q}G$
 realizing the outer action of $G$ on $M$
 gives rise to a crossed module $A\to A\times M\to\hat{G}\xrightarrow{q}G$
 providing two morphisms of four term exact sequences linking
 $A\to M\to N\to G$ to the trivial crossed module
 $A\to A \xrightarrow{0}G\to G$.

 Therefore we focus here on surjectivity, i.e. we construct a crossed
 module from a given locally continuous $3$-cocycle. For this, embed the
 $G$-module $A$ in a soft $G$-module:
 \begin{equation*}
  0\to A\to E_G(A)\to B_G(A)\to 0.
 \end{equation*}
 Observe that $H_{\SM}^n(G,E_G(A))\cong H^{n}_{\globc}(G,E_{G}(A))$
 vanishes for $n\geq 1$ (Proposition
 \ref{prop:soft_modules_are_acyclic_for_globcont}). The vanishing shows
 now that the connecting homomorphism of the associated long exact
 sequence induces an isomorphism
 \begin{equation*}
  \delta:H_{\locc}^2(G,B_G(A))\cong H_{\locc}^3(G,A),
 \end{equation*}
 where we have used the isomorphism of $H_{\SM}^{n}$ and
 $H_{\locc}^{n}$. Thus for the given $3$-cocycle $\gamma$, there exists
 a locally continuous $2$-cocycle $\alpha$ with values in $B_G(A)$ such
 that $\delta[\alpha]=[\gamma]$. Using $\alpha$, we can form an abelian
 extension
 \begin{equation*}
  0\to B_G(A)\to B_G(A)\times_{\alpha} G\to G\to 1.
 \end{equation*}
 Now splicing together this abelian extension with the short exact
 coefficient sequence
 \begin{equation*}
  0\to A\to E_G(A)\to B_G(A)\to 0
 \end{equation*}
 gives rise to a crossed module $\mu:E_G(A)\to B_G(A)\times_{\alpha} G$
 which is topologically split in the above sense. Indeed, the
 coefficient sequence is topologically split by assumption, and the
 abelian extension has a continuous local section by construction.

 Finally, the fact that the $3$-class associated to this crossed module
 is $[\gamma]$ follows from $\delta[\alpha]=[\gamma]$. Some details for
 this kind of construction can also be found in
 \cite{Wagemann06On-Lie-algebra-crossed-modules}.
\end{proof}

\begin{remark}
 In the case of locally compact second countable $G$ and metrizable $A$
 the module $EA$ is metrizable
 \cite{BrownMorris77Embeddings-in-contractible-or-compact-objects} and
 since $G$ is in particular $\sigma$-compact $C(G,EA)=E_{G}(A)$ is also
 metrizable. Thus the above crossed module is a crossed module of
 metrizable topological groups. In particular, if we take a generator
 $[\alpha]\in H^{3}_{SM}(G,U(1))\cong H^{4}(G,\Z)\cong \Z$ for $G$ a
 simple compact 1-connected Lie group, then the crossed module
 \begin{equation*}
  U(1) \to E_{G}(U(1)) \to B_{G}(U(1))\times_{\alpha}G \to G
 \end{equation*}
 gives yet another (topological) model for the string 2-group.
\end{remark}

\begin{remark}
 (cf.\ \cite[Def.\ 19]{Pries09Smooth-group-cohomology}) The locally
 continuous cohomology can be topologized as follows. For an open
 identity neighborhood $U\se G^{\ktimes[n]}$ we have the bijection
 \begin{equation*}
  C_{U}^{n}(G,A):=\{f\from G^{n}\to A:\left.f\right|_{U}\text{ is continuous}\}\cong
  C(U,A)\times A^{G^{n}\setminus U}.
 \end{equation*}
 This carries a natural topology coming from
 $\kC(U,A)\ktimes A^{G^{n}\setminus U}$, when first endowing
 $A^{G^{n}\setminus U}$ with the product topology and then taking the
 induced compactly generated topology. If $U\se V$, then the inclusion
 $C_{U}^{n}(G,A)\hookrightarrow C^{n}_{V}(G,A)$ is continuous so that
 the direct limit
 \begin{equation*}
  \lim_{\overrightarrow{U\in \mf{U}}} C_{U}^{n}(G,A) \cong C_{\locc}^{n}(G,A)
 \end{equation*}
 carries a natural topology. The differential $\dgp$ is continuous and
 the cohomology groups $H^{n}_{\locc}(G,A)$ inherits the corresponding
 quotient topology.
\end{remark}

\begin{remark}\label{rem:products}
 There is a classical way of constructing {\it products} for some of the
 cohomology theories which we have considered here. Let us recall these
 definitions. The easiest product is the usual {\it cup product} for the
 locally continuous (respectively the locally smooth) group cohomology
 $H^{n}_{\locc}(G,A)$ (respectively $H^{n}_{\locs}(G,A)$) \cite[Ch.\
 VIII.9]{MacLane63Homology}. In the following, we will stick to
 $H^{n}_{\locc}(G,A)$, noting that all constructions carry over word by
 word to $H^{n}_{\locs}(G,A)$ for a Lie group $G$ and a smooth
 $G$-module $A$.

 Suppose that the two $G$-modules $A$ and $A'$ have a tensor product in
 $\cghaus$. The simplicial cup product (see \cite{MacLane63Homology}
 equation (9.7) p. 246) in group cohomology yields a homomorphism
 \begin{equation*}
  \cup:H^p_{\locc}(G,A)\otimes H^q_{\locc}(G,A')\to H^{p+q}_{\locc}(G,A\otimes A'),
 \end{equation*}
 where the $G$-module $A\otimes A'$ is given the diagonal action.

 In case the $G$-module $A$ has its tensor product $A\otimes A$ in
 $\cghaus$ and has a product, i.e. a homomorphism of $G$-modules
 $\alpha:A\otimes A\to A$, we obtain an {\it internal cup product}
 \begin{equation*}
  \cup:H^p_{\locc}(G,A)\otimes H^q_{\locc}(G,A)\to H^{p+q}_{\locc}(G,A)
 \end{equation*}
 by post composing with $\alpha$. The product reads then explicitly for
 cochains $c\in C^p_{\locc}(G,A)$ and $c'\in C^q_{\locc}(G,A)$
 \begin{equation*}
  c\cup c'(g_0,\ldots,g_{p+q})=\alpha(c(g_0,\ldots,g_{p}),c'(g_{p},\ldots, g_{p+q})).
 \end{equation*}
 
 On the other hand, Segal-Mitchison cohomology $H^{n}_{\SM}(G,A)$ is a
 (relative) derived functor, and therefore the setting of \cite[Sect.\
 XII.10]{MacLane63Homology} is adapted. Observe that our choice of exact
 sequences does not satisfy all the demands of a {\it proper class} of
 exact sequences \cite[Sect.\ XII.4]{MacLane63Homology} (it does not
 satisfy the last two demands) and we neither have automatically enough
 proper injectives or projectives. Nevertheless, we have explicit
 acyclic resolutions for each module in $\cghaus$ which are exact
 sequences in our sense. We have the universality property for the
 functor $H^{n}_{\SM}(G,A)$ \cite[Sect.\ XII.8]{MacLane63Homology} by
 Theorem \ref{thm:moores_comparison_theorem}. Therefore we obtain
 products for Segal-Mitchison cohomology by universality as in
 \cite[Th.\ XII.10.4]{MacLane63Homology} for two $G$-modules $A$ and
 $A'$ which have a tensor product in $\cghaus$.

 By the uniqueness statement in \cite[Th.\ XII.10.4]{MacLane63Homology},
 the isomorphism $H^{n}_{\SM}(G,A)\cong H^{n}_{\locc}(G,A)$ respects
 products. Note also that the differentiation homomorphism
 $D_{n}\from H^{n}_{\locs}(G,A)\to H^{n}_{\op{Lie},c}(\fg,\mf{a})$ that we will turn
 to in Remark \ref{rem:connection_to_Lie_algebra_cohomology} is
 compatible with products.
\end{remark}

\begin{tabsection}
 We now give an explicit description of the purely topological
 information contained in a locally continuous cohomology class. If $G$
 is a connected topological group and $A$ is a topological $G$-module,
 then there is an exact sequence
 \begin{equation}\label{eqn:tau2}
  0 \to H^{2}_{\globtop}(G,A)\to  H^{2}_{\loctop}(G,A)\xrightarrow{\tau_{2}} \check{H}^{1}(|G|,\underline{A})
 \end{equation}
 \cite[Sect.\
 2]{Wockel09Non-integral-central-extensions-of-loop-groups}, where
 $\tau_2$ assigns to an abelian extension $A\to \hat{G}\to G$ the
 characteristic class of the underlying principal $A$-bundle. By
 definition, we have that $\im(\tau_2)$ are those classes in
 $\check{H}^{1}(|G|,\underline{A})$ whose associated principal $A$-bundles
 admit a compatible group structure.

 We will now establish a similar behavior of the map $\tau_{n}$ for
 arbitrary $n$.
\end{tabsection}

\begin{proposition}\label{prop:tau_is_delta_functor}
 Let $G$ be a connected topological group and $A$ be a topological
 $G$-module. Suppose that the cocycle $f\in C^{n}_{\loctop}(G,A)$ is
 continuous on the identity neighborhood $U\se G^{n}$ and let $V\se G$ be open
 such that $e\in V$ and $V^{2}\times \ldots\times V^{2}\se U$. Then the
 map
 \begin{multline*}
  \tau(f)_{g_{1},\ldots,g_{n}}\from g_{1}V\cap \ldots\cap g_{n}V\to A,\quad
  x\mapsto 
  g_{1}^{~}.		f(g_{1}^{-1}g_{2}^{~},\ldots,g_{n-1}^{-1}g_{n}^{~},g_{n}^{-1}x) 
  -(-1)^{n}	f(g_{1}^{~},g_{1}^{-1}g_{2}^{~},\ldots,g_{n-1}^{-1}g_{n}^{~})
 \end{multline*}
 defines a continuous \v{C}ech $(n-1)$-cocycle on the open cover
 $(gV)_{g\in G}$. Moreover, this induces a well-defined map
 \begin{equation*}
  \tau_{n}\from H^{n}_{\loctop}(G,A)\to \check{H}^{n-1}(|G|,\underline{A}),\quad
  [f]\mapsto [\tau(f)]
 \end{equation*}
 which is a morphism of $\delta$-functors.
\end{proposition}

\begin{proof}
 We first note that $\tau(f)_{g_{1},\ldots,g_{n}}$ depends continuously
 on $x$. Indeed, the first term depends continuously on $x$ since
 $g_{1}V\cap \ldots\cap g_{n}V\neq \emptyset$ implies that
 $g_{k-1}^{-1}g_{k}^{~}\in V^{2}$ and $f$ is continuous on
 $V^{2}\times\ldots\times V^{2}$ by assumption. Since the second term
 does not depend on $x$, this shows continuity. Now the cocycle identity
 for $f$, evaluated on
 $(g_{1}^{~},g_{1}^{-1}g_{2}^{~},\ldots,g_{n-1}^{-1}g_{n}^{~},g_{n}^{-1}x)$,
 shows that $\tau(f)_{g_{1},\ldots,g_{n}}(x)$ may also be written as
 $(\check \delta (\kappa(f)))_{g_{1},\ldots,g_{n}}(x) $, where
 \begin{equation*}
  \kappa(f)_{g_{2},\ldots,g_{n}}(x):=f(g_{2}^{~},g_{2}^{-1}g_{3}^{~},\ldots,g_{n}^{-1}x).
 \end{equation*}
 Note that $\kappa(f)_{g_{2},\ldots,g_{n}}$ does not depend continuously
 on $x$ and thus the above assertion does not imply that $\tau(f)$ is a
 coboundary. However, $\check{\delta}^{2}=0$ now implies that $\tau(f)$
 is a cocycle.

 Clearly, the class $[\tau(f)]$ in
 $\check{H}^{n-1}(|G|,\underline{A})$ does not depend on the choice of $V$
 since another such choice $V'$ yields a cocycle given by the same
 formula on the refined cover $(g(V\cap V'))_{g\in G}$. Moreover, if $f$
 is a coboundary, i.e., $f=\dgp b$ for $b\in C^{n-1}_{\locc}(G,A)$
 (where we assume w.l.o.g. that $b$ is also continuous on
 $V^{2}\times \ldots\times V^{2}$), then we set
 \begin{equation*}
  \rho(b)_{g_{1},\ldots,g_{n-1}}(x):=
  g_{1}.	b(g_{1}^{-1}g_{2}^{~},\ldots,g_{n-1}^{-1}x)+(-1)^{n}
  b(g_{1}^{~},g_{1}^{-1}g_{2},\ldots,g_{n-2}^{-1}g_{n-1}^{~}).
 \end{equation*}
 As above, this defines a continuous function on
 $g_{1}V\cap \ldots\cap g_{n-1}V\neq \emptyset$ and thus a \v{C}ech
 cochain. A direct calculation shows that
 $\check{\delta}(\rho(f))=\tau(f)$ and thus that the class $[\tau(f)]$
 only depends on the class of $f$.

 We now turn to the second claim, for which we have to check that for
 each exact sequence $A\hookrightarrow B\xrightarrow{q} C$ of
 topological $G$-modules the diagram
 \begin{equation*}
  \xymatrix{
  H^{n}_{\locc}(G,C) 
  \ar[r]^{\delta_{n}}\ar[d]^{\tau_{n}}&
  H^{n+1}_{\locc}(G,A)
  \ar[d]^{\tau_{n+1}}\\
  \check{H}^{n-1}(|G|,\underline{C})
  \ar[r]^{{\delta}_{n-1}}&
  \check{H}^{n}(|G|,\underline{A})
  }
 \end{equation*}
 commutes. For this, we recall that $\delta_{n}$ is constructed by
 choosing for $[f]\in H^{n}_{\locc}(G,C)$ a lift
 $\wt{f}\from G^{n}\to B$ and then setting
 $\delta_{n}([f])=[\dgp \wt{f}]$. After possibly shrinking $V$, we can
 assume that $f$ is continuous on $V^{2}\times \ldots\times V^{2}$ ($n$
 factors) and that $\dgp \wt{f}$ is continuous on
 $V^{2}\times \ldots\times V^{2}$ ($n+1$ factors).

 Since $q$ is a homomorphism, $\wt{f}$ also gives rise to lifts
 \begin{equation*}
  \wt{\tau(f)}_{g_{1},\ldots,g_{n}}(x):=
  g_{1}^{~}.		\wt{f}(g_{1}^{-1}g_{2}^{~},\ldots,g_{n-1}^{-1}g_{n}^{~},g_{n}^{-1}x) 
  -(-1)^{n}	\wt{f}(g_{1}^{~},g_{1}^{-1}g_{2}^{~},\ldots,g_{n-1}^{-1}g_{n}^{~})
 \end{equation*}
 of $\tau(f)_{g_{1},\ldots,g_{n}}$, which obviously depends continuously
 on $x$ on $g_{1}V\cap \ldots \cap g_{n}V$. Thus we have that
 ${\delta}_{n-1}(\tau_{n}([f]))$ is represented by the \v{C}ech cocycle
 \begin{equation*}
  \check{\delta}(\wt{\tau(f)})_{g_{0},\ldots,g_{n}}.
 \end{equation*}
 On the other hand, $\tau_{n+1}(\delta_{n}([f]))$ is represented by
 $\tau(\dgp \widetilde{f})_{g_{0},\ldots,g_{n}}$, whose value on $x$ is
 given by
 \begin{gather*}
  g_{0}^{~}.		\dgp \widetilde{f}(g_{0}^{-1}g_{1}^{~},\ldots,g_{n-1}^{-1}g_{n}^{~},g_{n}^{-1}x) 
  -(-1)^{n+1}	\dgp \widetilde{f}(g_{0}^{~},g_{0}^{-1}g_{1}^{~},\ldots,g_{n-1}^{-1}g_{n}^{~})=\\
  g_{0}^{~}.\Big[g_{0}^{-1}g_{1}^{~}.\widetilde{f}(g_{1}^{-1}g_{2}^{~},\ldots,g_{n}^{-1}x)\pm\ldots
  \dashuline{+(-1)^{k}\widetilde{f}(g_{0}^{-1}g_{1}^{~},\ldots,g_{k-1}^{-1}g_{k+1}^{~},\ldots,g_{n}^{-1}x)}\pm\ldots\\
  +\underline{(-1)^{n+1} \widetilde{f}(g_{0}^{-1}g_{1}^{~},\ldots,g_{n-1}^{-1}g_{n}^{~})}\Big]
  -(-1)^{n+1}\Big[
  g_{0}^{~}.\underline{\widetilde{f}(g_{0}^{-1}g_{1}^{~},\ldots,g_{n-1}^{-1}g_{n}^{~})}
  \pm\ldots\\ \dashuline{-(-1)^{k}\widetilde{f}(g_{0}^{~},g_{0}^{-1}g_{1}^{~},\ldots,g_{k-1}^{-1}g_{k+1}^{~},\ldots,g_{n-1}^{-1}g_{n}^{~})}\pm\ldots
  +(-1)^{n+1}\widetilde{f}(g_{0}^{~},g_{0}^{-1}g_{1}^{~},\ldots,g_{n-2}^{-1}g_{n-1}^{~})
  \Big]
 \end{gather*}
 The underlined terms cancel and the sum of the dashed terms gives
 $(-1)^{k}\wt{\tau(f)}_{g_{0},\ldots,\widehat{g_{k}},\ldots,g_{n}}(x)$.
 This shows that
 \begin{equation*}
  \check{\delta}(\wt{\tau(f)})_{g_{1},\ldots,g_{n}}(x)  =  \tau(\dgp \widetilde{f})_{g_{1},\ldots,g_{n}}(x).
 \end{equation*}
\end{proof}

\begin{tabsection}
 We will now identify the map $\tau$ with one of the edge homomorphisms
 in the spectral sequence associated to $H^{n}_{\simpc}(G,A)$.
\end{tabsection}

\begin{proposition}\label{prop:tau_is_edge}
 For $n\geq 1$ the edge homomorphism of the spectral sequence
 \eqref{eqn:spectral_sequence} induces a homomorphism
 \begin{equation*}
  \op{edge}_{n+1}\from H^{n+1}_{\simpc}(G,A)\to H^{1+n}_{\simpc}(G,A)/\mathcal{F}^{2}H^{2+n}_{\simpc}(G,A)\cong E_{\infty}^{1,n}\to E_{1}^{1,n}\cong H^{n}_{\Sh}(G,\underline{A}),
 \end{equation*}
 where $\mathcal{F}$ denotes the standard column filtration (cf.\ Remark
 \ref{rem:double_complex_for_spectral_sequence}). If, moreover,
 $G^{\ptimes[n]}$ is compactly generated, paracompact and admits good
 covers for all $n\geq 1$ and $A$ is a topological $G$-module, then the
 diagram
 \begin{equation}\label{eqn:edge_morphism_commuting_square}
  \vcenter{\xymatrix{
  H^{n+1}_{\simpc}(G,A)
  \ar[d]^{\cong}\ar[rr]^(.54){\op{edge}_{n+1}}&&
  H^{n}_{\Sh}(G,\underline{A})
  \ar[d]^{\cong}\\
  H^{n+1}_{\locc}(G,A)
  \ar[rr]^{\tau_{n+1}}&&
  \check{H}^{n}(|G|,\underline{A})
  }}
 \end{equation}
 commutes.
\end{proposition}

\begin{proof}
 We first note that $H^{n}_{\loctop}(G,A)=H^{n}_{\locc}(G,A)$ under the
 above assumptions. Since $BG_{0}=\op{pt}$, we have
 $E_{1}^{0,q}=H^{q}_{\Sh}(\op{pt},\underline{A})=0$ for all $q\geq 1$
 and thus the edge homomorphism $E_{\infty}^{1,p}\to E_{1}^{1,p}$. Since
 we have $\mathcal{F}^{p}H^{p+q}_{\simpc}(G,A)=H^{p+q}_{\simpc}(G,A)$
 for $p=0,1$, $q\geq 1$ this gives the desired form of
 $\op{edge}_{q+1}$. Since this construction commutes with the connecting
 homomorphisms, it is a morphism of $\delta$-functors. Moreover, the
 isomorphism
 $H^{n}_{\Sh}(|G|,\underline{\mathinner{\:\cdot\:}})\cong \check{H}^{n}(|G|,\underline{\mathinner{\:\cdot\:}})$
 is even an isomorphism of $\delta$-functors. In virtue of the
 uniqueness assertion for morphisms of $\delta$-functors from Theorem
 \ref{thm:moores_comparison_theorem}, it thus remains to verify that
 that \eqref{eqn:edge_morphism_commuting_square} commutes for $n=1$.

 The construction from Remark
 \ref{rem:morphism_from_locally_continuous_to_Cech_cohomology} gives an
 isomorphism
 $H^{2}_{\loctop}(G,A)\cong \check{H}^{2}(\mc{U}_{\bullet},A_{\locc}^{\bullet})$,
 where $\mc{U}_{\bullet}$ is a good cover of $BG_{\bullet}$ chosen such
 that $\mc{U}_{k}$ refines the covers of $G^{k}$ constructed there.
 Since this construction commutes with the connecting homomorphisms, the
 isomorphism
 $H^{2}_{\loctop}(G,A)\cong \check{H}^{2}(\mc{U}_{\bullet},A_{\locc}^{\bullet})$
 is indeed the one from the unique isomorphism of the corresponding
 $\delta$-functors. Now $\tau_{2}$ coincides with the morphism
 $H^{2}_{\loctop}(G,A)\cong \check{H}^{2}(\mc{U}_{\bullet},A_{\locc}^{\bullet})\to \check{H}^{1}(|G|,A)$,
 given by projecting the cocycle $(\mu,\tau)$ in the total complex of
 $\check{C}^{p,q}(\mc{U}_{\bullet},E^{\bullet})$ to the \v{C}ech cocycle
 $\tau$. Since this is just the corresponding edge homomorphism, the
 diagram \eqref{eqn:edge_morphism_commuting_square} commutes for $n=1$.
\end{proof}

\begin{remark}
 In case the action of $G$ on $A$ is trivial, Proposition
 \ref{prop:tau_is_edge} also holds for $n=0$. Indeed, then the
 differential $A\cong E_{1}^{0,0}\to E_{1}^{1,0}\cong C^{\infty}(G,A)$,
 which is given by assigning the principal crossed homomorphism to an
 element of $A$, vanishes. This shows commutativity of
 \eqref{eqn:edge_morphism_commuting_square} also for $n=0$.
\end{remark}

\begin{remark}
 The other edge homomorphism is induced from the identification
 $ C^{n}_{\globc}(G,A)\cong H^{0}_{\Sh}(G^{n},A)\cong E_{1}^{n,0}$,
 which shows $E_{2}^{n,0}\cong H^{n}_{\globc}(G,A)$. It coincides with
 the morphism $H^{n}_{\globc}(G,A)\to H^{n}_{\locc}(G,A)$ induced by the
 inclusion $C^{n}_{\globc}(G,A)\hookrightarrow C^{n}_{\locc}(G,A)$ (cf.\
 also \cite[Remarks in \S3]{Segal70Cohomology-of-topological-groups}).
\end{remark}

The following is a generalization of \eqref{eqn:tau2} in case $A$ is
discrete.

\begin{corollary}\label{cor:injectivity_of_tau}
 Suppose that $n\geq 1$, $G$ is $(n-1)$-connected, $A$ is a discrete $G$-module
 and that $G^{\ptimes[m]}$ is compactly generated, paracompact and admits good
 covers for all $m\geq 1$. Then
 $\tau_{n+1}\from H_{\locc}^{n+1}(G,A)\to \check{H}^{n}(|G|,\underline{A})$ is
 injective.
\end{corollary}

\begin{proof}
 If $G$ is $(n-1)$-connected, and $A$ is discrete, then $E_{1}^{p,q}$ of
 the spectral sequence \eqref{eqn:spectral_sequence} vanishes if
 $q\leq n-1$. Thus
 $E_{\infty}^{1,n-1}=\ker(d_{1}^{1,n-1})\se E_{1}^{1,n-1}\cong \check{H}^{n-1}(|G|,\underline{A})$
 and $\op{edge}_{n}$ coincides with the embedding
 \begin{equation*}
  H^{n}_{\locc}(G,A)\cong H^{n}_{\simpc}(G,A)\cong E_{\infty}^{1,n-1}\hookrightarrow E_{1}^{1,n-1}\cong \check{H}^{n-1}(|G|,\underline{A}).
 \end{equation*}
\end{proof}

\begin{remark}\label{rem:string_cocycle_2}
 An explicit analysis of the differentials of the spectral sequence
 \eqref{eqn:spectral_sequence} shows that for discrete $A$ with trivial
 $G$-action and $(n-1)$-connected $G$ the image of
 $\tau_{n+1}\from H^{n+1}_{\locc}(G,A)\to \check{H}^{n}(|G|,\underline{A})$
 consists of those cohomology classes $c\in\check{H}^{n}(|G|,\underline{A})$
 which are \emph{primitive}, i.e., for which
 \begin{equation*}
  \pr_{1}^{*}c \otimes \pr_{2}^{*}c= \mu^{*}c.
 \end{equation*}
 Since the primitive elements generate the rational cohomology of a compact Lie
 group $G$ \cite[p.\ 167, Thm.\
 IV]{GreubHalperinVanstone73Connections-curvature-and-cohomology.-Vol.-II:-Lie-groups-principal-bundles-and-characteristic-classes},
 it follows that all non-torsion elements in the lowest cohomology degree are
 primitive in this case.
 
 In particular, if $G$ is a compact, simple and $1$-connected (thus
 automatically $2$-connected), the generator of
 $\check{H}^{2}(|G|,\underline{U(1)})\cong \check{H}^{3}(|G|,\underline{\Z})\cong \Z$
 is primitive and thus
 $\tau_{4}\from H^{4}_{\locc}(G,\Z)\to \check{H}^{3}(|G|,\underline{\Z})$ is an
 isomorphism. Since the diagram
 \begin{equation*}
  \xymatrix{
  H^{4}_{\locc}(G,\Z)\ar[r]^{\tau_{4}^{\Z}}\ar[d]^{\cong}&\check{H}^{3}(|G|,\underline{\Z})\ar[d]^{\cong}\\
  H^{3}_{\locc}(G,U(1)) \ar[r]^{\tau_{3}^{U(1)}}&\check{H}^{2}(|G|,\underline{U(1)})
  }
 \end{equation*}
 commutes by Proposition \ref{prop:tau_is_delta_functor}, this shows that
 $\tau_{3}^{U(1)}$ is also an isomorphism. Since the string class
 $[\sigma_{3}]$ from Example \ref{ex:string_cocycles} maps under $\tau_{3}$ to
 a generator
 \cite{BrylinskiMcLaughlin93A-geometric-construction-of-the-first-Pontryagin-class,Chatzigiannis12A-generator-in-degree-3-oft-he-Cech-cohomology-of-simple-compact-Lie-groups},
 this shows that $[\sigma_{3}]$ gives indeed a generator of
 $H^{3}_{\locc}(G,U(1))$, and $[\theta_{3}]$ gives a generator of
 $H^{4}_{\locc}(G,\Z)$.
\end{remark}

\begin{remark}\label{rem:connection_to_Lie_algebra_cohomology}
 One reason for the importance of \emph{locally smooth} cohomology is that it
 allows for a direct connection to Lie algebra cohomology and thus may be
 computable in algebraic terms. This relation is induced by the differentiation
 homomorphism
 \begin{equation*}
  H^{n}_{\locs}(G,A)\xrightarrow{D_{n}} H^{n}_{\op{Lie},c}(\fg,\mf{a}),
 \end{equation*}
 where $H^{n}_{\op{Lie},c}$ denotes the continuous Lie algebra cohomology,
 $\fg$ is the Lie algebra of $G$ and $\mf{a}$ the induced infinitesimal
 topological $\fg$-module (cf.\ \cite[Thm.\
 V.2.6]{Neeb06Towards-a-Lie-theory-of-locally-convex-groups} and \cite[App.\
 B]{Neeb04Abelian-extensions-of-infinite-dimensional-Lie-groups}).
 
 Suppose $G$ is finite-dimensional. Then the kernel of $D_{n}$ consists of
 those cohomology classes $[f]$ that are represented by cocycles vanishing on
 some neighborhood of the identity. For $\Gamma=\{0\}$ this follows directly
 from \cite{Swierczkowski71Cohomology-of-group-germs-and-Lie-algebras}, where
 it is shown that the differentiation homomorphism from the cohomology of
 \emph{locally defined} smooth group cochains to Lie algebra cohomology is an
 isomorphism. Thus if $[f]\in\ker(D_{n})$, then there exists a locally defined
 smooth map $b$ with $\dgp b-f=0$ wherever defined. Since we can extend $b$
 arbitrarily to a locally smooth cochain this shows the claim. In the case of
 nontrivial $\Gamma$ one may deduce the claim from the case of trivial
 $\Gamma$ since $\mf{a}$ and $A=\mf{a}/\Gamma$ are isomorphic as local Lie
 groups so that $A$-valued local cochains can always be lifted to
 $\mf{a}$-valued local cochains. If $A^{\delta}$ denotes $A$ with the discrete
 topology, then the isomorphism
 $H^{n}_{\pi_{1}(BG)}(BG,\underline{A^{\delta}})\cong H^{n}_{\locs}(G,A^{\delta})$
 from Remark \ref{rem:properties} induces an exact sequence
 \begin{equation*}
  H^{n}_{\pi_{1}(BG)}(BG,\underline{A^{\delta}})\to  H^{n}_{\locs}(G,A)\xrightarrow{D_{n}} H^{n}_{\op{Lie},c}(\fg,\mf{a})
 \end{equation*}
 (see also
 \cite{Neeb02Central-extensions-of-infinite-dimensional-Lie-groups,Neeb04Abelian-extensions-of-infinite-dimensional-Lie-groups}
 for an exhaustive treatment of $D_{2}$ for general infinite-dimensional $G$).
 From the van Est spectral sequence
 \cite{Est58A-generalization-of-the-Cartan-Leray-spectral-sequence.-I-II} it
 follows that if $G$ is $n$-connected (more general $G$ may be
 infinite-dimensional with split de Rham complex
 \cite{Beggs87The-de-Rham-complex-on-infinite-dimensional-manifolds}), then
 differentiation induces an isomorphism
 \begin{equation*}
  H^{n}_{\globs}(G,\mf{a})\to H^{n}_{\op{Lie},c}(\fg,\mf{a}).
 \end{equation*}
 For $G$ an $(n-1)$-connected Lie group this is not true any more, for instance
 the Lie algebra $3$-cocycle
 $\langle [\mathinner{\cdot},\mathinner{\cdot}],\mathinner{\cdot}\rangle$ from
 Example \ref{ex:string_cocycles} is nontrivial but $H^{3}_{\globs}(G,\R)$
 vanishes by \cite[Thm.\
 1]{Est55On-the-algebraic-cohomology-concepts-in-Lie-groups.-I-II} for compact
 and connected $G$.
 
 However, there exist integrating cocycles when considering \emph{locally
 smooth} cohomology: If $G$ is an $(n-1)$-connected finite-dimensional Lie
 group and $A\cong{\mathfrak a}/\Gamma$ is a finite-dimensional smooth module
 for $\mf{a}$ a finite-dimensional $G$-module and $\Gamma$ a discrete
 submodule, then
 $D_{n}\from H^{n}_{\locs}(G,A)\to H^{n}_{\op{Lie},c}(\mathfrak{g},\mathfrak{a})$
 is injective and its image consists of those cohomology classes $[\omega]$
 whose associated period homomorphism $\per_{[\omega]}$ \cite[Def.\
 V.2.12]{Neeb06Towards-a-Lie-theory-of-locally-convex-groups} has image in
 $\Gamma$. In fact, $H^{n}_{\locs}(G,A^{\delta})$ vanishes (by Corollary
 \ref{cor:injectivity_of_tau}), and thus $D_{n}$ is injective. Surjectivity of
 $D_{n}$ may be seen from the following standard integration argument. If
 $\omega$ is a Lie algebra $n$-cocycle, then the associated left-invariant
 $n$-form $\omega^{l}$ is closed \cite[Lem.\
 3.10]{Neeb02Central-extensions-of-infinite-dimensional-Lie-groups}. If we make
 the choices of $\alpha_{k}$ for $1\leq k\leq n$ as in Example
 \ref{ex:path-space_construction}, then
 \begin{equation*}
  \Omega(g_{1},\ldots,g_{n}):= \int_{\alpha_{n}(g_{1},\ldots,g_{n})}\omega^{l}
 \end{equation*}
 defines
 \begin{itemize}
  \item a locally smooth group cochain on $G$, since $\alpha_{n}$ depends
        smoothly on $(g_{1},\ldots,g_{n})$ on an identity neighborhood and the
        integral depends smoothly $\alpha_{n}(g_{1},\ldots,g_{n})$.
  \item a group cocycle, since
        \begin{equation*}
         \dgp \Omega (g_{0},\ldots,g_{n})=\int_{\dgp \alpha (g_{0},\ldots,g_{n})}\omega^{l}\in \per_{\omega}(\pi_{n}(G))\se \Gamma.
        \end{equation*}
 \end{itemize}
 A straightforward calculation, similar to the ones in
 \cite{Neeb02Central-extensions-of-infinite-dimensional-Lie-groups} or
 \cite{Neeb04Abelian-extensions-of-infinite-dimensional-Lie-groups} now shows
 that $D_{n}([\Omega])=[\omega]$. We expect that large parts of this remark can
 be generalized to arbitrary infinite-dimensional $G$ with techniques similar
 to those of
 \cite{Neeb02Central-extensions-of-infinite-dimensional-Lie-groups,Neeb02Central-extensions-of-infinite-dimensional-Lie-groups}.
\end{remark}

\section{\texorpdfstring{$\delta$}{δ}-Functors}
\label{sect:universal_delta_functors}

\begin{tabsection}
 In this section we recall the basic setting of (cohomological)
 $\delta$-functors (sometimes also called ``satellites''), as for
 instance exposed in \cite[Chap.\
 3]{CartanEilenberg56Homological-algebra}, \cite[Sect.\
 III.5]{Buchsbaum55Exact-categories-and-duality}, \cite[Sect.\
 2]{Grothendieck57Sur-quelques-points-dalgebre-homologique} or
 \cite[Sect.\
 4]{Moore76Group-extensions-and-cohomology-for-locally-compact-groups.-III}.
 It will be important that the arguments work in more general categories
 than abelian ones, the only thing one needs is a notion of short exact
 sequence.
\end{tabsection}

\begin{definition}
 A \emph{category with short exact sequences} is a category $\cat{C}$,
 together with a distinguished class of composable morphisms
 $A\to B\to C$. The latter are called a short exact sequence. A
 morphisms between $A\to B\to C$ and $A'\to B'\to C'$ consists of
 morphisms $A\to A'$, $B\to B'$ and $C\to C'$ such that the diagram
 \begin{equation*}
  \xymatrix{A\ar[r]\ar[d]&B\ar[r]\ar[d]&C\ar[d]\\A'\ar[r]&B'\ar[r]&C'}
 \end{equation*}
 commutes.

 A (cohomological) $\delta$-functor on a category with short exact
 sequences is a sequence of functors
 \begin{equation*}
  (H^{n}\from \cat{C}\to\cat{Ab})_{n\in\N_{0}}
 \end{equation*}
 such that for each exact $A\to B\to C$ there exist morphisms
 $\delta_{n}\from H^{n}(C)\to H^{n+1}(A)$ turning
 \begin{equation*}
  H^{0}(A)\to H^{0}(B)\to H^{0}(C)\xrightarrow{\delta_{0}} \cdots \xrightarrow{\delta_{n-1}}H^{n}(A)\to H^{n}(B)
  \to H^{n}(C)\xrightarrow{\delta_{n}} \cdots
 \end{equation*}
 into an exact sequence\footnote{Note that we do not require $H^{0}$ to
 be left exact.} and that for each morphism of exact sequences the
 diagram
 \begin{equation}\label{eqn:delta_functor}
  \xymatrix{
  H^{n}(C)\ar[r]^{\delta_{n}}\ar[d]&H^{n+1}(A)\ar[d]\\
  H^{n}(C')\ar[r]^{\delta_{n}}&H^{n+1}(A')}
 \end{equation}
 commutes. A morphisms of $\delta$-functors from $(H^{n})_{n\in\N_{0}}$
 to $(G^{n})_{n\in\N_{0}}$ is a sequence of natural transformations
 $(\varphi^{n}\from H^{n}\Rightarrow G^{n})_{n\in\N_{0}}$ such that for
 each short exact $A\to B\to C$ the diagram
 \begin{equation}\label{eqn:morphism_of_delta_functors}
  \xymatrix{
  H^{n}(C)\ar[r]^{\delta_{n}}\ar[d]^{\varphi^{n}_{C}}&
  H^{n+1}(A)\ar[d]^{\varphi_{A}^{n+1}}\\
  G^{n}(C)\ar[r]^{\delta_{n}}&G^{n+1}(A)}
 \end{equation}
 commutes. An isomorphism of $\delta$-functors is then a morphism for
 which all $\varphi^{n}$ are natural isomorphisms of functors.
\end{definition}

\begin{theorem}\label{thm:moores_comparison_theorem}
 Let $\cat{C}$ be a category with short exact sequences. Let
 $F\from \cat{C}\to\cat{Ab}$, $I\from \cat{C}\to\cat{C}$ and
 $U\from \cat{C}\to\cat{C}$ be functors, $\iota_{A}\from A\to I(A)$ and
 $\zeta_{A}\from I(A)\to U(A)$ be natural such that
 $A\xrightarrow{\iota_{A}} I(A)\xrightarrow{\zeta_{A}} U(A)$ is short
 sequence and let $(H_{n})_{n\in\N_{0}}$ and
 $(G_{n})_{n\in\N_{0}}$ be two $\delta$-functors.
 \begin{enumerate}
     \renewcommand{\labelenumi}{\theenumi}
     \renewcommand{\theenumi}{\arabic{enumi}.}
  \item \label{item:moores_comparison_theorem_1} If $\alpha\from H^{0}\Rightarrow G^{0}$ is a natural
        transformation and $H^{n}(I(A))=0$ for all $A$ and all
        $1\leq n\leq m$, then there exist natural transformations
        $\varphi^{n}\from H^{n}\Rightarrow G^{n}$, uniquely determined
        by requiring that $\varphi^{0}=\alpha$ and that
        \begin{equation*}
         \xymatrix{
         H^{n}(U(A)) \ar[r]^{\delta_{n}}\ar[d]_{\varphi^{n}_{U(A)}} & H^{n+1}(A) \ar[d]^{\varphi^{n+1}_{A}}\\
         G^{n}(U(A)) \ar[r]^{\overline{\delta}_{n}} & G^{n+1}(A)
         }
        \end{equation*}
        commutes for $0\leq n< m$. In particular, if
        $H^{n}(I(A))=0=G^{n}(I(A))$ for all $n\geq 0$, then
        $\varphi^{n}$ is an isomorphism of functors for all $n\in\N$ if
        and only if it is so for $n=0$.
  \item \label{item:moores_comparison_theorem_2} Assume, moreover, that for any short exact sequence
        $A\xrightarrow{f} B\to C$ the morphism $A\to I(B)$ can be
        completed to a short exact sequence $A\to I(B)\to Q_{f}$ such
        that there exist morphisms $U(A)\xrightarrow{\beta_{f}}Q_{f}$
        and $C\xrightarrow{\gamma_{f}} Q_{f}$ making
        \begin{equation}\label{eqn:moores_comparison_theorem}
         \vcenter{  \xymatrix{
         A\ar[r]^{\iota_{A}}\ar@{=}[d] & I(A) \ar[r]^{\zeta_{A}}\ar[d]^{I(f)} & U(A) \ar[d]^{\beta_{f}}\\
         A\ar[r] & I(B)\ar[r] & Q_{f}
         } }\quad\text{ and }\quad
         \vcenter{  \xymatrix{
         A\ar[r]^{f}\ar@{=}[d] & B \ar[r]\ar[d]^{\iota_{B}} & C \ar[d]^{\gamma_{f}}\\
         A\ar[r] & I(B)\ar[r] & Q_{f}
         } }
        \end{equation}
        commute. Then the diagram
        \begin{equation*}
         \xymatrix{
         H^{n}(C) \ar[r]^{\delta_{n}}\ar[d]_{\varphi^{n}_{C}} & H^{n+1}(A) \ar[d]^{\varphi^{k}_{A}}\\
         G^{n}(C) \ar[r]^{\overline{\delta}_{n}} & G^{n+1}(A)
         }
        \end{equation*}
        also commutes for $0\leq m<m$. In particular, if $H^{n}(I(A))=0$
        for all $A$ and all $n\geq 1$, then $(\varphi^{n})_{n\in\N_{0}}$
        is a morphism of $\delta$-functors.
 \end{enumerate}
\end{theorem}

\begin{proof}
 The proof of \cite[Thm.\
 II.5.1]{Buchsbaum55Exact-categories-and-duality} (cf.\ also \cite[Thm.\
 2]{Moore76Group-extensions-and-cohomology-for-locally-compact-groups.-III})
 carries over to this more general setting. The claims are shown by
 induction, so we assume that $\varphi ^{n}$ is constructed up to
 $n\geq 0$. Then we consider for arbitrary $A$ the diagram (recall that 
 $H^{n+1}(I(A))=0$)
 \begin{equation*}
  \vcenter{\xymatrix{
  H^{n}(I(A)) 
  \ar[r] \ar[d]^{\varphi^{n}_{I(A)}} &
  H^{n}(U(A))
  \ar[r]^{\delta_{n}} \ar[d]^{\varphi^{n}_{U(A)}} &
  H^{n+1}(A) 
  \ar[r] &
  0\\
  G^{n}(I(A))
  \ar[r] &
  G^{n}(U(A))
  \ar[r]^{\overline{\delta}_{n}} &
  G^{n+1}(A)
  }},
 \end{equation*}
 which shows that there is a unique
 $\varphi ^{n+1}_{A}\from H^{n+1}(A)\to G^{n+1}(A)$ making this diagram
 commute. To check naturality take $f\from A\to B$. By the construction
 of $\varphi_{A}^{n+1}$, the induction hypothesis and the construction
 of $\varphi_{B}^{n+1}$ the diagrams
 \begin{equation*}
  \vcenter{ \xymatrix{
  H^{n}(U(A))
  \ar[d]^{\varphi_{U(A)}^{n}}\ar[r]^{\delta_{n}^{U(A)}} &
  H^{n+1}(A)
  \ar[d]^{\varphi_{A}^{n+1}}\\
  G^{n}(U(A))
  \ar[r]^{\overline{\delta}_{n}^{U(A)}} &
  G^{n+1}(A)
  }}
  \quad\text{ and }\quad
  \vcenter{\xymatrix{
  H^{n}(U(A))
  \ar[d]^{\varphi_{U(A)}^{n}}\ar[rr]^{H^{n}(U(f))} &&
  H^{n}(U(B))
  \ar[d]^{\varphi_{U(B)}^{n}}\ar[r]^{\delta_{n}^{U(B)}} &
  H^{n+1}(B)
  \ar[d]^{\varphi_{B}^{n+1}}\\
  G^{n}(U(A))
  \ar[rr]^{{G^{n}(U(f))}} &&
  G^{n}(U(B))
  \ar[r]^{\overline{\delta}_{n}^{U(B)}} &
  G^{n+1}(B)
  }}
 \end{equation*}
 commute. Since $(H_{n})_{n\in\N_{0}}$ and $(G_{n})_{n\in\N_{0}}$ are
 $\delta$-functors we know that
 $H^{n+1}(f)\op{\circ} \delta_{n}^{U(A)}= \delta_{n}^{U(B)} \op{\circ} H^{n}(U(f)) $
 and that
 $G^{n+1}(f)\op{\circ} \overline{\delta}_{n}^{U(A)}= \overline{\delta}_{n}^{U(B)} \op{\circ} G^{n}(U(f)) $.
 We thus conclude that
 \begin{equation*}
  \varphi_{B}^{n+1}\op{\circ} H^{n+1}(f) \op{\circ} \delta_{n}^{U(A)}=
  G^{n+1}(f) \op{\circ} \varphi_{A}^{n+1} \op{\circ} \delta_{n}^{U(A)}
 \end{equation*}
 holds. Since $\delta_{n}^{U(A)}$ is an epimorphism this shows
 naturality of $\varphi^{n+1}$ and finishes the proof of the first
 claim.

 To show the second claim we note that the first diagram of
 \eqref{eqn:moores_comparison_theorem} gives rise to a diagram
 \begin{equation*}
  \vcenter{\xymatrix@=1.5em{
  H^{n}(U(A)) 
  \ar[ddd]_{\varphi_{U(A)}^{n}}
  \ar[dr]^(0.6){H^{n}(\beta_{f})} 
  \ar@/^15pt/[drr]^{\delta_{n}^{U(A)}}\\
  & H^{n}(Q_{f})
  \ar[r]^{\delta_{n}^{Q_{f}}}
  \ar[d]_{\varphi_{Q_{f}}^{n}}&
  H^{n+1}(A) 
  \ar[r]
  \ar[d]^{\varphi_{A}^{n+1}}& 
  0\\
  & G^{n}(Q_{f})
  \ar[r]^{\overline{\delta}_{n}^{Q_{f}}} &
  G^{n+1}(A)\\
  G^{n}(U(A))
  \ar[ur]_(0.6){G^{n}(\beta_{f})}
  \ar@/_15pt/[urr]_{{\overline{\delta}_{n}^{U(A)}}}
  }}.
 \end{equation*}
 The outer diagram commutes by construction of $\varphi_{A}^{n+1}$ (see above),
 the already shown naturality of $\varphi^{n}$
 shows that the trapezoid on the left
 commutes and the two triangles are commutative because $H$ and $G$ are 
 $\delta$-functors. This implies that the whole diagram commutes. In particular,
 we have
 $\varphi_{A}^{n+1}\op{\circ}\delta_{n}^{Q_{f}}=\overline{\delta}_{n}^{Q_{f}}\op{\circ} \varphi_{Q_{f}}^{n}$.
 The latter now implies that
 \begin{equation*}
  \xymatrix{
  H^{n}(C)
  \ar[d]^{\delta_{n}^{C}}
  \ar[rr]^{H^{n}(\gamma_{f})} &&
  H^{n}(Q_{f})
  \ar[d]^{\delta_{n}^{Q_{f}}}
  \ar[r]^{\varphi^{n}_{Q_{f}}} &
  G^{n}(Q_{f})
  \ar[d]^{\overline{\delta}_{n}^{Q_{f}}}
  &&
  G^{n}(C)
  \ar[d]^{\overline{\delta}_{n}^{C}}
  \ar[ll]_{G^{n}(\gamma_{f})}\\
  H^{n+1}(A)
  \ar@{=}[rr] &&
  H^{n+1}(A)
  \ar[r]^{\varphi_{A}^{n+1}} &
  G^{n+1}(A)
  &&
  G^{n}(A)
  \ar@{=}[ll]
  }
 \end{equation*}
 commutes and since
 $G^{n}(\gamma_{f})\op{\circ}\varphi_{C}^{n}=\varphi_{Q_{f}}^{n}\op{\circ} H^{n}(\gamma_{f})$
 we eventually conclude that
 \begin{equation*}
  \overline{\delta}_{n}^{C}\op{\circ}\varphi_{C}^{n}=
  \overline{\delta}_{n}^{Q_{f}}\op{\circ}G^{n}(\gamma_{f})\op{\circ} \varphi_{C}^{n}=
  \overline{\delta}_{n}^{Q_{f}}\op{\circ} \varphi_{Q_{f}}^{n}\op{\circ} H^{n}(\gamma_{f})=
  \varphi_{A}^{n+1}\op{\circ}\delta_{n}^{Q_{f}}\op{\circ}H^{n}(\gamma_{f})=
  \varphi_{A}^{n+1}\op{\circ}\delta_{n}^{C}.
 \end{equation*}
\end{proof}

\begin{remark}\label{rem:weaker_comparison_theorem}
 The preceding theorem also shows the following slightly stronger
 statement. Assume that we have for each $\delta$-functor
 $H=(H^{n})_{n\in\N_{0}}$ and $G=(G^{n})_{n\in\N_{0}}$ (defined on the
 \emph{same} category with short exact sequences) \emph{different}
 functors $I$, $U$ and $I'$, $U'$ such that $H^{n}(I(A))=0=G^{n}(I'(A))$
 for all $n\geq 1$ and all $A$. Suppose that the assumptions of Theorem
 \ref{thm:moores_comparison_theorem}
 (\ref{item:moores_comparison_theorem_2}) are fulfilled for one of the
 functors $I$ or $I'$.

 If $\alpha\from H^{0}\rightarrow G^{0}$ is an isomorphism, then the
 natural transformations $\varphi^{n}\from H^{n}\Rightarrow G^{n}$
 (resulting from extending $\alpha$) and
 $\psi^{n}\from G^{n}\Rightarrow H^{n}$ (resulting from extending
 $\alpha^{-1}$) are in fact isomorphisms of $\delta$-functors. This
 follows immediately from the uniqueness assertion since the diagrams
 \begin{equation*}
  \vcenter{\xymatrix{
  H^{n}(U(A)) \ar[r]^{\delta_{n}}\ar[d]_{\varphi^{n}_{U(A)}} & H^{n+1}(A) \ar[d]^{\varphi^{k}_{A}}\\
  G^{n}(U(A)) \ar[r]^{\overline{\delta}_{n}} & G^{n+1}(A)
  }}
  \quad\quad\quad
  \vcenter{\xymatrix{
  G^{n}(U(A)) \ar[r]^{\overline{\delta}_{n}}\ar[d]_{\psi^{n}_{U(A)}} & G^{n+1}(A) \ar[d]^{\psi^{k}_{A}}\\
  H^{n}(U(A)) \ar[r]^{{\delta}_{n}} & H^{n+1}(A)
  }}
 \end{equation*}
 (and likewise for $U'$) commute for arbitrary $A$ due to the property
 of being a $\delta$-functor.
\end{remark}

\begin{remark}
 Usually, one would impose some additional conditions on a category with
 short exact sequences, for instance that it is additive (with zero
 object), that for a short exact sequences $A\to B\to C$ the square
 \begin{equation*}
  \xymatrix{A\ar[r]\ar[d]&0\ar[d]\\B\ar[r]&C }
 \end{equation*}
 is a pull-back and a push-out, that short exact sequences are closed
 under isomorphisms and that certain pull-backs and push-outs exist
 \cite{Buhler10Exact-categories}. These assumptions will then help in
 constructing $\delta$-functors. However, the above setting does not
 require this, all the assumptions are put into the requirements on the
 $\delta$-functor.
\end{remark}

\begin{example}
 Suppose $G$ is paracompact. On the category of $G$-modules in $\cghaus$, we consider the short
 exact sequences $A\xrightarrow{\alpha}B\xrightarrow{\beta}C$ such that
 $\beta$ (or equivalently $\alpha$) has a continuous local section and
 the functor $A\mapsto \check{H}^{n}(|G|,\underline{A})$ (or equivalently
 $A\mapsto H^{n}_{\Sh}(G,\underline{A})$). Then the functors
 $A\mapsto E_{G}(A)$ and $A\mapsto B_{G}(A)$ from Definition
 \ref{def:segalsCohomology} satisfy $\check{H}^{n}(|G|,E_{G}(A))=0$ since
 $E_{G}(A)$ is contractible.
\end{example}

\begin{remark}\label{rem:counterexample_to_tus_comparison_theorem}
 The argument given in the proof of \cite[Prop.\
 6.1(b)]{Tu06Groupoid-cohomology-and-extensions} in order to draw the
 conclusion of the first part of Theorem
 \ref{thm:moores_comparison_theorem} from weaker assumptions is false as
 one can see as follows. First note that the proof only uses
 $I(U(A))\cong U(I(A))$, the more restrictive assumptions on the
 categories to be abelian and on the natural inclusion
 $A\hookrightarrow I(A)$ to satisfy $I(i_{A})=i_{I(A)}$ may be replaced
 by this.

 The requirements of \cite[Prop.\
 6.1(b)]{Tu06Groupoid-cohomology-and-extensions} are satisfied if we set
 $I(A)=E_{G}(A)$, $U(A)=B_{G}(A)$ and $i_{A}$ as in Definition
 \ref{def:segalsCohomology}. In fact, the exactness of the functor $E$
 shows that
 \begin{equation*}
  0\to EA\to E \kC(G,EA) \to E B_{G}(A)\to 0
 \end{equation*}
 is exact and since this sequence has a continuous section by
 \cite[Thm.\ B.2]{Segal70Cohomology-of-topological-groups}, we also have
 that
 \begin{equation*}
  0\to \kC(G,EA)\to \kC(G,E \kC(G,EA)) \to \kC(G,E B_{G}(A))\to 0 
 \end{equation*}
 is exact. Consequently, we have
 \begin{equation*}
  E_{G}(B_{G}(A))=\kC(G,E B_{G}(A))\cong \kC(G,E \kC(G,EA))/\kC(G,EA) =B_{G}(E_{G}(A)).
 \end{equation*}
 
 However, the two sequences of functors
 $A\mapsto H^{n}_{\SM}(G,A)\cong H^{n}_{\locc}(G,A)$ and
 $A\mapsto H^{n}_{\globc}(G,A)$ vanish on $E_{G}(A)$ for $n=1$, but are
 different:
 \begin{itemize}
  \item $H^{2}_{\globc}(G,A)$ is not isomorphic to $H^{2}_{\locc}(G,A)$,
        for instance for $G=C^{\infty}(\bS^{1},K)$ ($K$ compact, simple
        and 1-connected) and $A={U}(1)$.
  \item For non-simply connected $G$, the universal cover gives rise to
        a an element in $H^{2}_{\locc}(G,\pi_{1}(G))$, not contained in
        the image of $H^{2}_{\globc}(G,\pi_{1}(G))$.
  \item The string classes from Example \ref{ex:string_cocycles} gives
        an element in $H^{3}_{\locc}(K,{U}(1))$, not contained in the
        image of $H^{3}_{\globc}(K,{U}(1))$.
 \end{itemize}
 
\end{remark}

\section{Supplements on Segal-Mitchison cohomology}
\label{sect:some_information_on_moore_s_and_segal_s_cohomology_groups}

\begin{tabsection}
 We shortly recall the definition of the cohomology groups due to Segal
 and Mitchison from \cite{Segal70Cohomology-of-topological-groups}.
 Moreover, we also establish the acyclicity of the soft modules from
 above for the globally continuous group cohomology and show
 $H^{n}_{\SM}(G,A')\cong H^{n}_{\globc}(G,A')$ for contractible $A'$.
 Consider the long exact sequence
 \begin{equation}\label{eqn:resolution_for_Segal_cohomology}
  A\to E_{G}A\to E_{G}(B_{G}A)\to E_{G}(B_{G}^{2}A)\to E_{G}(B_{G}^{3}A) \cdots.
 \end{equation}
 This serves as a resolution of $A$ for the invariants functor
 $A\mapsto A^{G}$ and the cohomology groups $H^{n}_{\SM}(G,A)$ are the
 cohomology groups of the complex
 \begin{equation}\label{eqn:complex_for_Segal_cohomology}
  (E_{G}A)^{G}\to (E_{G}(B_{G}A))^{G}\to(E_{G}(B_{G}^{2}A))^{G}\to
  (E_{G}(B_{G}^{3}A))^{G}\cdots .
 \end{equation}
 
 We now make the following observations:
 \begin{enumerate}
        \renewcommand{\labelenumi}{\theenumi}
        \renewcommand{\theenumi}{\arabic{enumi}.}
  \item \label{item:soft_explanation1} \cite[Ex.\
        2.4]{Segal70Cohomology-of-topological-groups} For an arbitrary
        short exact sequence $\kC(G,A)\to B\to C$, the sequence
        \begin{equation*}
         \kC(G,A)^{G}\to B^{G}\to C^{G}
        \end{equation*}
        is exact, i.e., $B^{G}\to C^{G}$ is surjective. Indeed, for
        $y\in C^{G}$ choose an inverse image $x\in B$ and observe that
        $g.x-x$ may be interpreted as an element of $\kC(G,A)$ for each
        $g\in G$. If we define
        \begin{equation*}
         \psi(g,h):=(g.x-x)(h)\quad\text{ and }\quad
         \xi(h):=h.\psi(h^{-1},e)\footnote{Note that the leading $h$ is missing in \cite[Ex.\
         2.4]{Segal70Cohomology-of-topological-groups}.},
        \end{equation*}
        then we have $g.\xi-\xi=g.x-x$ since
        \begin{align*}
         (g.\xi-\xi)(h)=&g.(\xi(g^{-1}h))-\xi(h)=
         h.(\psi(h^{-1}g,e))-h.\psi(h^{-1},e)\\
         =&h.((h^{-1}.g.x-x)(e)-(h^{-1}.x-x)(e))\\
         =&h.((h^{-1}.(g.x-x))(e))=(g.x-x)(h).
        \end{align*}
        Thus $x-\xi$ is $G$-invariant and maps to $y$.
  \item It is not necessary to work with the resolution
        \eqref{eqn:resolution_for_Segal_cohomology}, any resolution
        \begin{equation}\label{eqn:alternative_resolution_for_Segal_cohomology}
         A\to A_{0}\to A_{1}\to A_{2}\to \cdots 
        \end{equation}
        (i.e., a long exact sequence of abelian groups such that the
        constituting short exact sequences have local continuous
        sections) with $A_{i}$ of the form $\kC(G,A'_{i})$ for some
        contractible $A'_{i}$ would do the job. Indeed, then the
        double complex
        \begin{equation*}
         \xymatrix@=0.75em{
         \vdots  & \vdots & \vdots &
         \vdots \\
         E_{G}(B_{G}^{2}A) \ar[r]\ar[u] & E_{G}(B_{G}^{2} (\kC(G,A'_{0})))\ar[r] \ar[u]& E_{G}(B_{G} ^{2}(\kC(G^{2},A'_{1})))\ar[r]\ar[u] &
         E_{G}(B_{G}^{2}( \kC(G^{3},A'_{2})))\ar[r]\ar[u]&\cdots\\
         E_{G}(B_{G}A) \ar[r]\ar[u] & E_{G}(B_{G} (\kC(G,A'_{0})))\ar[r] \ar[u]& E_{G}(B_{G} (\kC(G^{2},A'_{1})))\ar[r]\ar[u] &
         E_{G}(B_{G}( \kC(G^{3},A'_{2})))\ar[r]\ar[u]&\cdots\\
         E_{G}(A) \ar[r]\ar[u] & E_{G}(\kC(G,A'_{0}))\ar[r] \ar[u]& E_{G}(\kC(G^{2},A'_{1}))\ar[r]\ar[u] &
         E_{G}(\kC(G^{3},A'_{2}))\ar[r]\ar[u]&\cdots\\
         A \ar[r]\ar[u] & \kC(G,A'_{0})\ar[r] \ar[u]& \kC(G^{2},A'_{1})\ar[r]\ar[u] &
         \kC(G^{3},A'_{2})\ar[r]\ar[u]&\cdots
         }
        \end{equation*}
        has exact rows and columns (cf.\ \cite[Prop.\
        2.2]{Segal70Cohomology-of-topological-groups}), which remain
        exact after applying the invariants functor to it by the
        observation from \ref{item:soft_explanation1} Thus the
        cohomology of the first row is that of the first column, showing
        that the cohomology of \ref{eqn:complex_for_Segal_cohomology} is
        the same as the cohomology of
        $A_{0}^{G}\to A_{1}^{G}\to A_{2}^{G}\to\cdots$.

        In particular, for contractible $A'$ we may replace
        \eqref{eqn:resolution_for_Segal_cohomology} in the definition of
        $H_{\SM}^{n}(G,A')$ by
        \begin{equation*}
         A'\to E'_{G}A\to {E'}_{G}(B'_{G}A')\to E'_{G}({B'}_{G}^{2}A)\to E'_{G}({B'}_{G}^{3}A') \cdots
        \end{equation*}
        with $E'_{G}(A'):=\kC(G,A')$ and $B'_{G}(A):=E'_{G}(A)/A$ (the
        occurrence of $E$ in the definition $E_{G}(A):=\kC(G,EA)$ only
        serves the purpose of making the target contractible).
  \item \label{item:soft_explanation2} Since $A'$ is assumed to be
        contractible, the short exact sequence
        $A'\to E'_{G}(A')\to B'_{G}(A')$ has a global continuous section
        \cite[App.\ B]{Segal70Cohomology-of-topological-groups}, and
        thus the sequence
        \begin{equation*}
         \kC(G,A')\to \kC(G,E'_{G}(A'))\to \kC(G,B'_{G}(A'))
        \end{equation*}
        is exact. In particular, the isomorphism
        $\kC(G,E'_{G}(A'))\cong E'_{G}(\kC(G,A'))$ shows that
        \begin{equation*}
         B'_{G}(\kC(G,A')):=E'_{G}(\kC(G,A'))/\kC(G,A')\cong \kC(G,E'_{G}(A'))/\kC(G,A')\cong \kC(G,B'_{G}(A'))
        \end{equation*}
        is again of the form $\kC(G,A'')$ with $A''$ contractible.
 \end{enumerate}
 These observations, together with an inductive argument, imply that the
 sequence
 \begin{equation*}
  A^{G}\to (E'_{G}A)^{G}\to (E'_{G}(B_{G}A))^{G}\to(E_{G}({B'}_{G}^{2}A))^{G}\to
  (E'_{G}({B'}_{G}^{3}A))^{G}\cdots 
 \end{equation*}
 is exact for $A=\kC(G,A')$ and contractible $A'$, and finally that
 $H^{n}_{\SM}(G,A)$ vanishes for $n\geq 1$.  What also follows is that
 for contractible $A'$, we have
 $H^{n}_{\SM}(G,A')\cong H^{n}_{\globc}(G,A')$ (cf.\ \cite[Prop.\
 3.1]{Segal70Cohomology-of-topological-groups}). Indeed,
 $\kC(G^{k},A')\cong \kC(G,\kC(G^{k-1},A'))$ and thus
 \begin{equation*}
  A' \to \kC(G,A')\to \kC(G^{2},A')\to \kC(G^{3},A')\to \cdots
 \end{equation*}
 serves as a resolution of the form
 \eqref{eqn:alternative_resolution_for_Segal_cohomology}. Dropping $A'$
 and applying the invariants functor to it then gives the (homogeneous
 version of) the complex $C^{n}_{\globc}(G,A')$.
\end{tabsection}

\bibliographystyle{new} \def\polhk#1{\setbox0=\hbox{#1}{\ooalign{\hidewidth
  \lower1.5ex\hbox{`}\hidewidth\crcr\unhbox0}}}

\end{document}